\documentclass[11pt]{amsart}
\usepackage[margin=1in]{geometry}
\usepackage{latexsym}

\usepackage{amsfonts}
\usepackage{amsmath}
\usepackage{amssymb}
\usepackage{amsthm}
\usepackage{enumerate}
\usepackage[hang,flushmargin]{footmisc}
\usepackage{caption}
\usepackage{tabu}
\usepackage{mathrsfs}
\usepackage{amsaddr}
\usepackage{subfig}
\usepackage{mathtools}

\usepackage{xcolor}

\definecolor{MyBlue}{rgb}{0,0,1}
\definecolor{MyRed}{rgb}{1,0,0}
\definecolor{MyGreen}{rgb}{0,1,0}
\definecolor{MyIndigo}{rgb}{0.7254,0,1}
\definecolor{MyOrange}{rgb}{1,0.4431,0}

\usepackage{yhmath}
\usepackage{graphicx}
\usepackage{epstopdf}
\usepackage{epsfig}
\usepackage{caption}

\DeclareFontFamily{U}{mathx}{\hyphenchar\font45}
\DeclareFontShape{U}{mathx}{m}{n}{
      <5> <6> <7> <8> <9> <10>
      <10.95> <12> <14.4> <17.28> <20.74> <24.88>
      mathx10
      }{}
\DeclareSymbolFont{mathx}{U}{mathx}{m}{n}
\DeclareFontSubstitution{U}{mathx}{m}{n}
\DeclareMathAccent{\widecheck}{0}{mathx}{"71}
 
\usepackage{bm} 
\usepackage[noadjust]{cite}

\makeatletter

\newtheorem{theorem}{Theorem}[section]
\newtheorem{lemma}[theorem]{Lemma}
\newtheorem{proposition}[theorem]{Proposition}
\newtheorem{corollary}[theorem]{Corollary}

\newtheorem{conjecture}[theorem]{Conjecture}

\theoremstyle{definition}
\newtheorem{definition}[theorem]{Definition}
\newtheorem{remarkx}[theorem]{Remark}
\newenvironment{remark}
  {\pushQED{\qed}\remarkx}
  {\popQED\endremarkx}

\DeclareMathOperator{\Pop}{\mathsf{Pop}}
\DeclareMathOperator{\SP}{SP}
\DeclareMathOperator{\rev}{rev}
\DeclareMathOperator{\half}{half}
\DeclareMathOperator{\cyc}{cyc}

\newcommand{\dfn}[1]{\textcolor{blue}{\emph{#1}}}

\begin{document}
\title{Pop-Stack-Sorting for Coxeter Groups}
\author{Colin Defant}
\address{Princeton University \\ Department of Mathematics \\ Princeton, NJ 08544}
\email{cdefant@princeton.edu}

\begin{abstract}
Let $W$ be an irreducible Coxeter group. We define the \emph{Coxeter pop-stack-sorting operator} $\Pop:W\to W$ to be the map that fixes the identity element and sends each nonidentity element $w$ to the meet of the elements covered by $w$ in the right weak order. When $W$ is the symmetric group $S_n$, $\Pop$ coincides with the pop-stack-sorting map. Generalizing a theorem about the pop-stack-sorting map due to Ungar, we prove that \[\sup\limits_{w\in W}\left|O_{\Pop}(w)\right|=h,\] where $h$ is the Coxeter number of $W$ (with $h=\infty$ if $W$ is infinite) and $O_f(w)$ denotes the forward orbit of $w$ under a map $f$. When $W$ is finite, this result is equivalent to the statement that the maximum number of terms appearing in the Brieskorn normal form of an element of $W$ is $h-1$. More generally, we define a map $f:W\to W$ to be \emph{compulsive} if for every $w\in W$, $f(w)$ is less than or equal to $\Pop(w)$ in the right weak order. We prove that if $f$ is compulsive, then $\sup\limits_{w\in W}|O_f(w)|\leq h$. This result is new even for symmetric groups. 

We prove that $2$-pop-stack-sortable elements in type $B$ are in bijection with $2$-pop-stack-sortable permutations in type $A$, which were enumerated by Pudwell and Smith. Claesson and Gu{\dh}mundsson proved that for each fixed nonnegative integer $t$, the generating function that counts $t$-pop-stack-sortable permutations in type $A$ is rational; we establish analogous results in types~$B$ and~$\widetilde A$.  
\end{abstract}


\maketitle

\bigskip

\section{Introduction}\label{Sec:Intro} 

\subsection{Sorting Operators}
\emph{Noninvertible combinatorial dynamics} is the study of combinatorially-defined dynamical systems on sets of combinatorial objects, where emphasis is placed on understanding the transient (i.e., non-periodic) points. Given an arbitrary set $X$ and a map $f:X\to X$, let $f^t$ denote the $t^\text{th}$ iterate of $f$. The \dfn{forward orbit} of an element $x\in X$ under the map $f$ is the set $O_f(x)=\{x,f(x),f^2(x),\ldots\}$. It is natural to consider the quantity $\sup\limits_{x\in X}\left|O_f(x)\right|$. Indeed, if $f$ is invertible and all orbits are finite, this is equivalent to determining the largest size of a periodic orbit. A similar question asks for the maximum possible number of iterations needed to send every transient point to a periodic point. If $f$ has a fixed point $x_0$ such that every forward orbit under $f$ contains $x_0$ (as will be the case for all of the dynamical systems considered in this paper), then these two questions are essentially equivalent. 

The symmetric group $S_n$, which is the group of permutations of the set $[n]=\{1,\ldots,n\}$, provides a rich source of noninvertible combinatorial dynamical systems. We can write a permutation $w\in S_n$ in one-line notation as $w(1)\cdots w(n)$. A great amount of research in combinatorics and computer science has focused on sorting operators, which are dynamical systems on $S_n$ that have the identity permutation $e=123\cdots n$ as their unique periodic (necessarily fixed) point. Some typical examples of such operators include the bubble sort map (see \cite{Chung, AlbertBubble} and \cite[pages 106--110]{Knuth2}), West's stack-sorting map (see \cite{BonaSurvey, DefantCounting, DefantMonotonicity, DefantTroupes, West} and the references therein), the map $\text{revstack}$ defined in \cite{Dukes}, the pop-stack-sorting map \cite{AlbertVatter, Asinowski, Asinowski2, Elder, ClaessonPop, ClaessonPop2, Pudwell, Ungar}, and the $\mathtt{Queuesort}$ map \cite{Cioni, Magnusson}. 

A natural condition one might wish to place on a sorting operator $f:S_n\to S_n$, which all five of the specific sorting operators listed in the previous paragraph satisfy, is that it does not remove non-inversions. This means that if $a,b\in[n]$ are such that $a<b$ and $a$ appears to the left of $b$ in $w$, then $a$ must appear to the left of $b$ in $f(w)$. Among the sorting operators $f$ that do not remove non-inversions, we will be primarily interested in the ones that remove consecutive inversions, meaning that if $a,b\in[n]$ are such that $a<b$ and $a$ appears immediately to the right of $b$ in $w$, then $a$ appears to the left of $b$ in $f(w)$. West's stack-sorting map, the map $\text{revstack}$, and the pop-stack-sorting map all remove consecutive inversions; the bubble sort map and $\mathtt{Queuesort}$, however, do not. 

We are actually interested in generalizing these notions to arbitrary Coxeter groups. The condition that $f:S_n\to S_n$ does not remove non-inversions is equivalent to the condition that $f(w)\leq_R w$ for all $w\in S_n$, where $\leq_R$ is the right weak order on $S_n$. Saying that $f$ does not remove non-inversions \emph{and} removes consecutive inversions is equivalent to saying that $f(w)\leq_R w$ and $f(w)\leq_R ws$ for every $w\in W$ and every right descent $s$ of $w$. In what follows, we let $\leq_R$ denote the right weak order on an arbitrary Coxeter group $W$, and we let $D_R(w)$ denote the right descent set of an element $w\in W$ (see Section~\ref{Subsec:Coxeter} for definitions).

\begin{definition}\label{DefCox1}
Let $W$ be a Coxeter group. We say a map $f:W\to W$ is \dfn{compulsive}\footnote{Google gives the following two definitions for the word \emph{compulsive}: (1) resulting from or relating to an irresistible urge, especially one that is against one's conscious wishes; (2) irresistibly interesting or exciting; compelling. Our motivation for using this word comes from the first definition since a compulsive map on $S_n$ compulsively removes all consecutive inversions. However, we hope to convince the reader that the second definition is also appropriate.} if $f(w)\leq_R w$ and $f(w)\leq_R ws$ for every $w\in W$ and every $s\in D_R(w)$.  
\end{definition}

Note that the condition $f(w)\leq_R w$ is necessary in Definition~\ref{DefCox1} in order to guarantee that $f$ fixes the identity element $e$. 

A seminal result due to Bj\"orner \cite{Bjorner} states that the right weak order on a Coxeter group $W$ is a complete meet-semilattice. This means that every set $A\subseteq W$ has a greatest lower bound, called the \dfn{meet} of $A$, which we denote by $\bigwedge_RA$. Hence, a map $f:W\to W$ is compulsive if and only if $f(w)\leq_R\bigwedge_R(\{ws:s\in D_R(w)\}\cup\{w\})$ for every $w\in W$. This motivates the following definition. 

\begin{definition}
Let $W$ be a Coxeter group. The \dfn{Coxeter pop-stack-sorting operator} on $W$ is the map $\Pop_W:W\to W$ defined by \[\Pop_W(w)=\bigwedge\nolimits_R(\{ws:s\in D_R(w)\}\cup\{w\})\] for every $w\in W$. 
\end{definition}

We often write $\Pop$ instead of $\Pop_W$ if the group $W$ is clear from context. 

The Coxeter pop-stack-sorting operator is certainly compulsive; its name comes from the fact, which we will verify in Section~\ref{Subsec:PopMap}, that the Coxeter pop-stack-sorting operator on $S_n$ is precisely the pop-stack-sorting map. 
This map, which is a deterministic analogue of a pop-stack-sorting machine introduced by Avis and Newborn in \cite{Avis}, first appeared in a different guise in a paper of Ungar's about discrete geometry \cite{Ungar}; its popularity has grown rapidly over the past few years \cite{AlbertVatter, Asinowski, Asinowski2, Elder, Pudwell, ClaessonPop, ClaessonPop2}. 
Ungar \cite{Ungar}, motivated by a question involving directions determined by points in the plane, proved that the maximum possible size of a forward orbit of a permutation in $S_n$ under the pop-stack-sorting map is $n$; this settled a conjecture due to Goodman and Pollack \cite{Goodman}. In other words, Ungar's theorem states that $\Pop^{n-1}(w)=e$ for every $w\in S_n$ and that there exists $v\in S_n$ such that $\Pop^{n-2}(v)\neq e$. The proof requires an unexpected amount of insight. Recently, Albert and Vatter \cite{AlbertVatter} provided an alternative proof of Ungar's theorem. 

Our first main result generalizes Ungar's theorem to an arbitrary irreducible\footnote{The assumption of irreducibility does not limit the scope of the theorem. If $W=W_1\times W_2$ is reducible, then $\Pop$ acts on $W_1$ and $W_2$ independently, and one can understand the dynamics of $\Pop$ on $W$ by ``piecing together'' information about the dynamics of $\Pop$ on $W_1$ and $W_2$.} Coxeter group. Let $(W,S)$ be a Coxeter system. If $W$ is finite, then a \dfn{Coxeter element} of $W$ is an element obtained by multiplying the simple generators (the elements of $S$) in an arbitrary order. All Coxeter elements have the same order in the group $W$; this order is called the \dfn{Coxeter number} of $W$ and is typically denoted by $h$. For example, the Coxeter number of $S_n$ is $n$. We make the convention that the Coxeter number of an infinite Coxeter group is $\infty$.  

\begin{theorem}\label{ThmCox1}
If $W$ is an irreducible Coxeter group with Coxeter number $h$, then \[\sup_{w\in W}\left|O_{\Pop}(w)\right|=h.\] 
\end{theorem}

The preceding theorem tells us that if $W$ is finite, then $\Pop^{h-1}(w)=e$ for all $w\in W$ and $\Pop^{h-2}(v)\neq e$ for some $v\in W$. On the other hand, if $W$ is infinite, this theorem says that there are arbitrarily large forward orbits of elements of $W$ under $\Pop$. Our proof of Theorem~\ref{ThmCox1} is Coxeter-theoretic and is mostly type-independent in the sense that it avoids the use of combinatorial models of finite Coxeter groups. However, we must treat symmetric groups and dihedral groups separately from the other finite irreducible Coxeter groups; this stems from the fact that all finite irreducible Coxeter groups that are not symmetric groups or dihedral groups have even Coxeter numbers. 

\begin{figure}[ht]
  \begin{center}{\includegraphics[height=4cm]{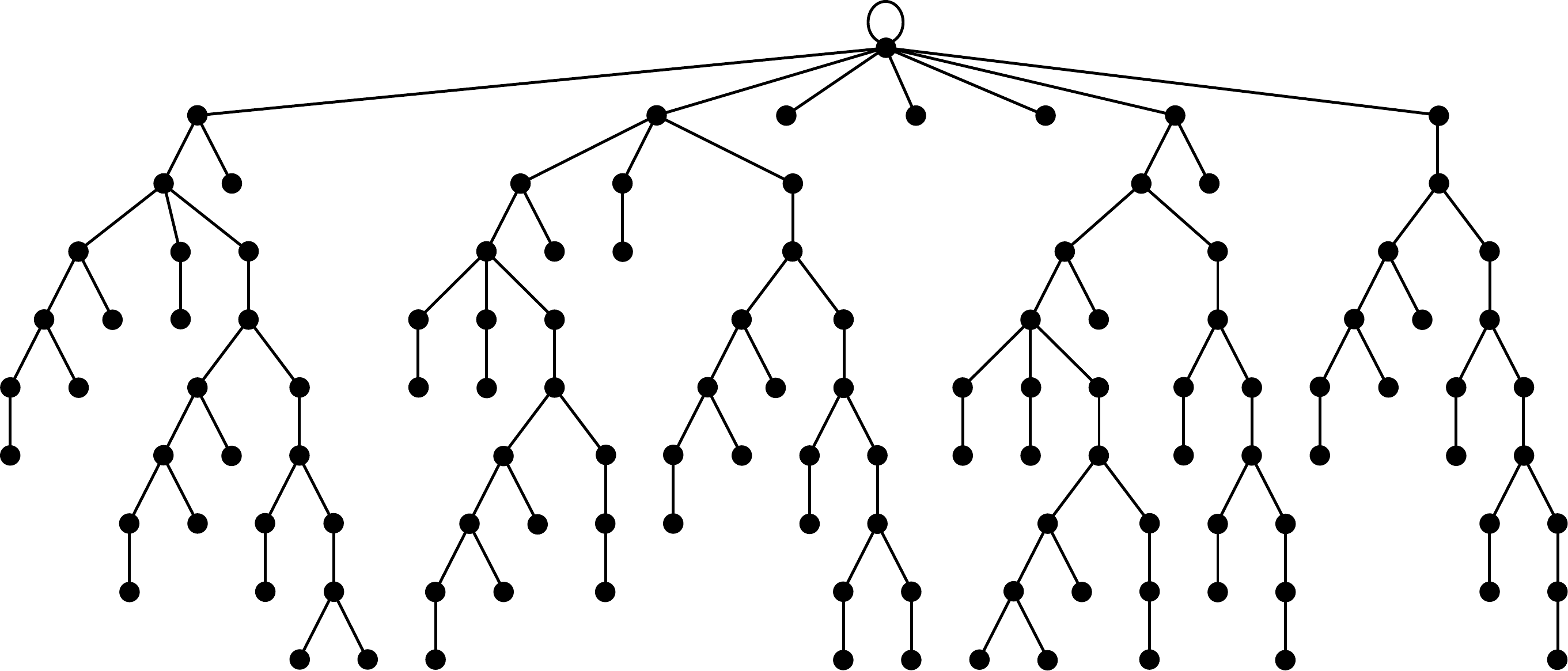}}
  \end{center}
  \caption{A diagram of the Coxeter pop-stack-sorting operator $\Pop:H_3\to H_3$. Each vertex in the tree represents an element of $H_3$ (we omit the labels). The root vertex is the identity element $e$; the parent of each non-root vertex $w$ is $\Pop(w)$. As predicted by Theorem~\ref{ThmCox1}, the maximum size of a forward orbit is $10$, which is the Coxeter number of $H_3$.}\label{FigCox1}
\end{figure}

It turns out that $\Pop$ is very closely related to \emph{Brieskorn normal form}, which was introduced in the foundational works of Brieskorn and Brieskorn--Saito on complex reflection groups and Artin groups \cite{Brieskorn,Brieskorn2}. This normal form is a specific factorization of an element of $W$ into longest elements of parabolic subgroups. It follows from equation \eqref{EqCox1} below that the elements of $O_{\Pop}(w)$ are exactly the prefixes of the Brieskorn normal form of $w$.\footnote{The articles \cite{Brieskorn,Brieskorn2} use left weak order where we use right weak order, but this makes little difference.} More precisely, the Brieskorn normal form of $w$ is $\beta_1\cdots\beta_r$, where $r=|O_{\Pop}(w)|-1$ and $\beta_1\cdots\beta_i=\Pop^{r-i}(w)$ for all $0\leq i\leq r$. Hence, the following corollary is an immediate consequence of Theorem~\ref{ThmCox1}.

\begin{corollary}
If $W$ is a finite irreducible Coxeter group with Coxeter number $h$, then the maximum number of terms that can appear in the Brieskorn normal form of an element of $W$ is $h-1$.
\end{corollary}

Our second main result applies to all compulsive maps on an irreducible Coxeter group. This theorem is new even for symmetric groups.

\begin{theorem}\label{ThmCox2}
Let $W$ be an irreducible Coxeter group with Coxeter number $h$. If $f:W\to W$ is compulsive, then \[\sup_{w\in W}\left|O_f(w)\right|\leq h.\] 
\end{theorem}

The conclusion of Theorem~\ref{ThmCox2} is trivial if $W$ is infinite, so we will only need to consider the case when $W$ is finite. Our proof of this result is delicate and makes use of the Coxeter pop-stack-sorting operator, so we will really need to deliver it simultaneously with the proof of Theorem~\ref{ThmCox1}. 

\begin{remark}
It is perhaps tempting to think that Theorem~\ref{ThmCox2}
would follow as an immediate corollary of Theorem~\ref{ThmCox1} since a compulsive map $f:W\to W$ satisfies $f(w)\leq_R\Pop(w)$ for all $w\in W$ by definition. More precisely, one might think that $f^t(w)\leq_R\Pop^t(w)$ for all $w\in W$ and $t\geq 0$ so that, by Theorem~\ref{ThmCox1}, we have (assuming $W$ is finite) $f^{h-1}(w)\leq_R\Pop^{h-1}(w)=e$ for all $w\in W$. However, this is not the case. If $W=S_5$ and $\mathtt{s}$ is West's stack-sorting map (see \cite{BonaSurvey, DefantCounting} for the definition), then $\mathtt{s}^3(42351)=21345$, but $\Pop^3(42351)=12345=e$.  
\end{remark}

Another typical approach to noninvertible combinatorial dynamical systems, especially sorting operators, concerns the enumeration of elements that require at most some fixed number $t$ of iterations to reach a periodic point. In the case of West's stack-sorting map, these elements are the \emph{$t$-stack-sortable} permutations, which have been studied extensively, especially for $t\in\{1,2,3\}$ (see \cite{Albert, Branden3, BonaSurvey, BonaSymmetry, DefantCounting, DefantTroupes, DefantElvey} and the references therein). For the pop-stack-sorting map, these elements are the \emph{$t$-pop-stack-sortable} permutations investigated in \cite{Pudwell, ClaessonPop, Elder}. Given a Coxeter group $W$, let us say an element $w\in W$ is \dfn{$t$-pop-stack-sortable} if $\Pop_W^t(w)=e$. In other words, the set of $t$-pop-stack-sortable elements of $W$ is $\Pop_W^{-t}(e)$. At the end of Section~\ref{Subsec:Coxeter}, we will establish the following simple proposition. The \emph{nerve} of Coxeter system is defined in that section. 

\begin{proposition}\label{PropCox2}
Let $(W,S)$ be a Coxeter system with nerve $\mathcal N(W,S)$. There is a bijection between the set $\Pop_{W}^{-1}(e)$ of $1$-pop-stack-sortable elements of $W$ and $\mathcal N(W,S)$. In particular, if $W$ is finite, then the number of $1$-pop-stack-sortable elements of $W$ is $2^{|S|}$. 
\end{proposition}

Pudwell and Smith \cite{Pudwell} enumerated $2$-pop-stack-sortable permutations in $S_n$. Our next theorem provides an analogue of this result for the hyperoctahedral groups $B_n$ (defined in Section~\ref{Sec:TypeB}). Descending runs are defined in Section~\ref{Subsec:PopMap}. 

\begin{theorem}\label{ThmCox3} 
For each $n\geq 1$ and $k\geq 0$, the number of $2$-pop-stack-sortable elements of $B_n$ with $2k$ or $2k+1$ descending runs is equal to the number of $2$-pop-stack-sortable permutations in $S_{n+1}$ with exactly $k+1$ descending runs. In particular, \[\sum_{n\geq 1}\left|\Pop_{B_n}^{-2}(e)\right|z^n=\sum_{n\geq 1}\left|\Pop_{S_{n+1}}^{-2}(e)\right|z^n=\frac{2z(1+z+z^2)}{1-2z-z^2-2z^3}.\]
\end{theorem}

Claesson and Gu{\dh}mundsson proved that for every fixed $t\geq 0$, the generating function that counts $t$-pop-stack-sortable permutations in symmetric groups is rational \cite{ClaessonPop}. Our final main theorems extend this result to the hyperoctahedral groups $B_n$ and the affine symmetric groups $\widetilde S_n$ (see Sections~\ref{Sec:TypeB} and \ref{SecAffine} for the definitions). 

\begin{theorem}\label{ThmCox4}
For every $t\geq 0$, the generating function $\displaystyle \sum_{n\geq 1}\left|\Pop_{B_n}^{-t}(e)\right|z^n$ is rational. 
\end{theorem}

\begin{theorem}\label{ThmCox5}
For every $t\geq 0$, the generating function $\displaystyle \sum_{n\geq 1}\left|\Pop_{\widetilde S_n}^{-t}(e)\right|z^n$ is rational. 
\end{theorem}

\subsection{Generalizations}\label{Subsec:Further}

There are (at least) two natural ways in which one could generalize the definition of the Coxeter pop-stack-sorting operators. 

First, one could replace the right weak order on a Coxeter group with an arbitrary complete meet-semilattice $M$ to obtain a map $\Pop_M:M\to M$ defined by \[\Pop_M(x)=\bigwedge(\{y\in M:y\lessdot x\}\cup\{x\})\] for all $x\in M$. This provides a large new class of noninvertible combinatorial dynamical systems that are ripe for investigation.\footnote{Henri M\"uhle has informed the author that $\Pop_M(x)$ is the same as what he has called the \emph{nucleus} of $x$ \cite{Muhle1, Muhle2}. However, his results do not overlap with ours because he did not view nuclei as defining a dynamical system. Instead, he was interested in the lattice-theoretic properties of the interval between $\Pop_M(x)$ and $x$ in $M$.} We will explore this avenue in \cite{DefantMeeting}, with special emphasis on $\nu$-Tamari lattices.   

\begin{remark}\label{RemCox1}
In light of Definition~\ref{DefCox1}, it is natural to define a map $f:M\to M$ on an arbitrary complete meet-semilattice $M$ to be \dfn{compulsive} if $f(x)\leq \Pop_M(x)$ for all $x\in M$. Together, Theorems~\ref{ThmCox1} and \ref{ThmCox2} tell us that if $M$ is isomorphic to the right weak order on an irreducible Coxeter group and $f:M\to M$ is compulsive, then $\sup\limits_{x\in M}\left| O_f(x)\right|\leq \sup\limits_{x\in M}\left| O_{\Pop_M}(x)\right|$. This does \emph{not} hold if $M$ is replaced by an arbitrary complete meet-semilattice. For example, consider the lattice $M$ whose Hasse diagram is shown in Figure~\ref{FigCox2}. If $f:M\to M$ is the compulsive map illustrated with green arrows, then $\sup\limits_{x\in M}\left| O_f(x)\right|=4>3=\sup\limits_{x\in M}\left| O_{\Pop_M}(x)\right|$. 
\end{remark}

\begin{figure}[ht]
  \begin{center}{\includegraphics[height=3cm]{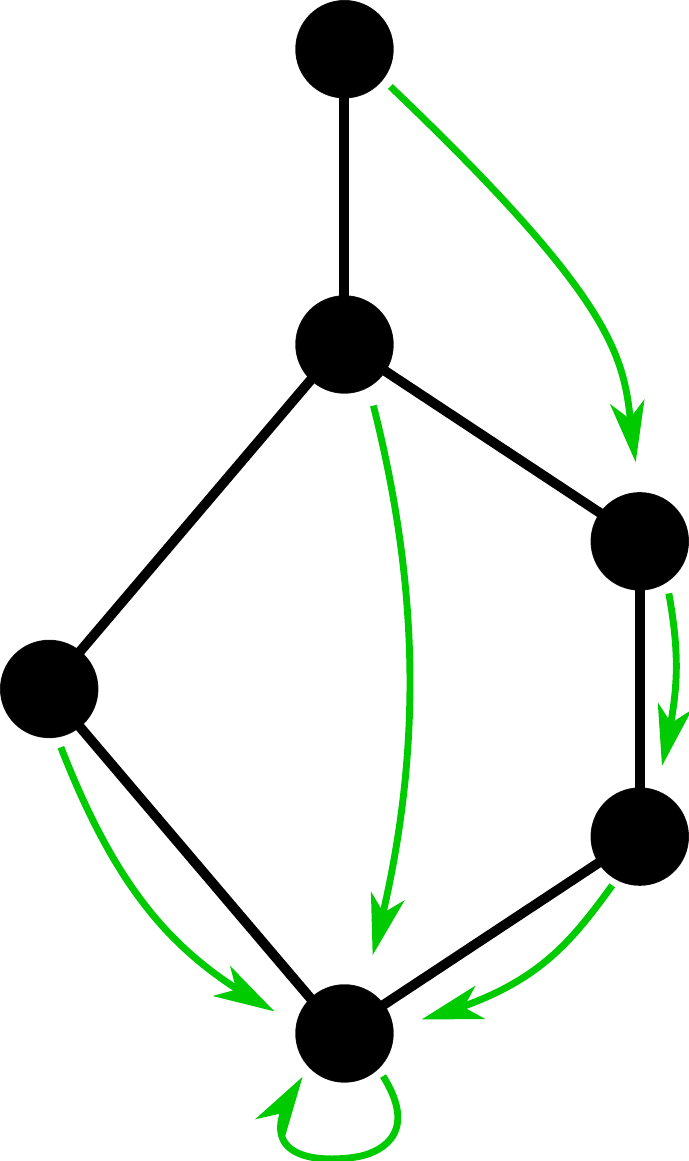}}
  \end{center}
  \caption{A compulsive map on a lattice.}\label{FigCox2}
\end{figure}

For the second generalization, we make use of the left weak order $\leq_L$ on a Coxeter group $W$. The poset $(W,\leq_L)$ is a complete meet-semilattice that is isomorphic to $(W,\leq_R)$. Let $\wedge_L$ denote the meet operation in the left weak order. A \dfn{semilattice congruence} on $(W,\leq_L)$ is an equivalence relation $\equiv$ on $W$ that respects meets, meaning that $x_1\equiv x_2$ and $y_1\equiv y_2$ together imply $(x_1\wedge_L y_1)\equiv (x_2\wedge_L y_2)$. When $W$ is finite, we define a \dfn{lattice congruence} on $(W,\leq_L)$ to be an equivalence relation on $W$ that respects both meets and joins, meaning $x_1\equiv x_2$ and $y_1\equiv y_2$ together imply $(x_1\wedge_L y_1)\equiv (x_2\wedge_L y_2)$ and $(x_1\vee_L y_1)\equiv (x_2\vee_L y_2)$. Semilattice congruences and lattice congruences on weak orders of Coxeter groups have been studied extensively (see \cite{Pilaud, Pilaud2, Pilaud3, Law, Giraudo, Hoang, ReadingCambrian, ReadingLattice2, ReadingFinite, ReadingSpeyerSortable, ReadingSpeyerCambrian, ReadingSpeyerFrameworks} and the references therein). 

One of the most natural examples of a semilattice congruence on $(W,\leq_L)$ is the \dfn{descent congruence}, which is defined by declaring two elements of $W$ to be equivalent if they have the same right descent set. Another notable example is the \dfn{sylvester congruence} on the symmetric group $S_n$. There is a well-known bijection $\mathcal I$ from the set of decreasing binary plane trees with label set $[n]$ to $S_n$; the sylvester congruence is defined by declaring two permutations $\sigma$ and $\sigma'$ to be equivalent if the unlabeled binary plane trees obtained by removing the labels from $\mathcal I^{-1}(\sigma)$ and $\mathcal I^{-1}(\sigma')$ are equal (see \cite{DefantPolyurethane, Hivert, Pilaud} for more details). Further interesting examples of semilattice congruences on Coxeter groups are provided by the Cambrian congruences (see \cite{ReadingCambrian, ReadingFinite, ReadingSpeyerCambrian, ReadingSpeyerSortable} and the references therein). In the case of $S_n$, other notable examples of lattice congruences include the permutree congruences \cite{Pilaud}, the $k$-twist congruences \cite{Pilaud3}, and the Baxter congruence \cite{Law, Giraudo}. 

If $\equiv$ is a semilattice congruence on the left weak order of $W$, then every congruence class of $\equiv$ has a unique minimal element. We denote by $\pi_\downarrow:W\to W$ the projection map that sends each element of $W$ to the unique minimal element of its congruence class. We say $\equiv$ is \dfn{essential} if the identity element $e\in W$ is in a singleton equivalence class.

\begin{definition}
Let $W$ be a Coxeter group, and let $\equiv$ be an essential semilattice congruence on the left weak order of $W$. Define the \dfn{Coxeter stack-sorting operator} ${\bf S}_\equiv:W\to W$ to be the map given by ${\bf S}_\equiv(w)=w(\pi_\downarrow(w))^{-1}$ for all $w\in W$.   
\end{definition}

The name \emph{Coxeter stack-sorting operator} is motivated by two special cases. First, if $\equiv$ is the descent congruence, then ${\bf S}_\equiv$ is the same as the Coxeter pop-stack-sorting operator $\Pop_W$ (we justify this claim at the end of Section~\ref{Subsec:Coxeter}). Second, if $W=S_n$ and $\equiv$ is the sylvester congruence, then ${\bf S}_\equiv$ is West's stack-sorting map (this is essentially the content of \cite[Corollary~16]{DefantPolyurethane}, where the projection map $\pi_\downarrow$ goes by the name $\text{swd}$). Coxeter stack-sorting operators provide another broad new class of non-invertible combinatorial dynamical systems; we will investigate them more extensively in \cite{DefantCoxeterStack}, with special emphasis on permutree congruences in type~$A$ and their analogues (in particular, analogues of the sylvester congruence) in types~$B$ and~$\widetilde A$. 

\begin{remark}
A semilattice congruence on the left weak order of a Coxeter group is essential if and only if it refines the descent congruence (see \cite{Hoang, DefantCoxeterStack}). From this, it is not difficult to show that all Coxeter stack-sorting operators are compulsive, so they provide a wide range of maps to which one can apply Theorem~\ref{ThmCox2} (see \cite{DefantCoxeterStack}).  
\end{remark}

\begin{remark}
A different generalization of West's stack-sorting map, introduced in \cite{Cerbai}, uses \emph{pattern-avoiding stacks}; this notion has attracted much attention in recent years \cite{Baril, Berlow, Cerbai3, Cerbai2, DefantZheng}. While these pattern-avoiding stacks are certainly interesting, we believe our Coxeter stack-sorting operators are more natural from an algebraic and lattice-theoretic point of view. 
\end{remark}

\subsection{Outline}
In Section~\ref{Sec:Preliminaries}, we recall relevant definitions and facts concerning Coxeter groups, prove Proposition~\ref{PropCox2}, and verify that the Coxeter pop-stack-sorting operator on $S_n$ coincides with the pop-stack-sorting map. Section~\ref{SecMaximal} is devoted to proving Theorems~\ref{ThmCox1} and \ref{ThmCox2}. We prove Theorems~\ref{ThmCox3} and \ref{ThmCox4}, which concern Coxeter groups of type $B$, in Section~\ref{ThmCox4}. We prove Theorem~\ref{ThmCox5}, which concerns Coxeter groups of type $\widetilde A$, in Section~\ref{SecAffine}. Section~\ref{SecConclusion} provides additional potential ideas for future work.

\section{Preliminaries}\label{Sec:Preliminaries}

\subsection{Coxeter Groups}\label{Subsec:Coxeter}

We assume familiarity with basic notions and concepts from the combinatorial theory of Coxeter groups and from lattice theory; a standard reference that contains all the background information we need is \cite{BjornerBrenti}.

A \dfn{Coxeter system} is a pair $(W,S)$, where $W$ is a group generated by the set $S$ with presentation $W=\langle S:(ss')^{m(s,s')}=e\rangle$ such that $m(s,s)=1$ for all $s\in S$ and $m(s,s')=m(s',s)\in\{2,3,\ldots\}\cup\{\infty\}$ for all distinct $s,s'\in S$. The elements of $S$ are called the \dfn{simple generators}. We will often refer to a Coxeter group $W$ with the understanding that we are really referring to a Coxeter system $(W,S)$ for some specific choice of a generating set $S$. The \dfn{Coxeter diagram} of $W$ is the graph $\Gamma(W)$ with vertex set $S$ in which two vertices $s,s'$ are connected by an edge labeled with $m(s,s')$ if $m(s,s')\geq 3$ (and are not adjacent if $m(s,s')\leq 2$). Notice that $s$ and $s'$ commute if and only if they are not adjacent in $\Gamma(W)$. We say $W$ is \dfn{irreducible} if $\Gamma(W)$ is a connected graph. 

A \dfn{reduced word} for an element $w\in W$ is a word $s_1\cdots s_k$ over the alphabet $S$ that, when viewed as a product of elements of $W$, equals $w$. The smallest length of a reduced word for $w$ is called the \dfn{length} of $w$ and is denoted by $\ell(w)$. The \dfn{left weak order} on $W$ is the partial order $\leq_L$ on $W$ defined by saying $x\leq_L y$ if $\ell(yx^{-1})=\ell(y)-\ell(x)$. The \dfn{right weak order} on $W$ is the partial order $\leq_R$ on $W$ defined by saying $x\leq_R y$ if $\ell(x^{-1}y)=\ell(y)-\ell(x)$. The map $W\to W$ given by $w\mapsto w^{-1}$ is an isomorphism from the left weak order to the right weak order. We will also need the \dfn{strong Bruhat order} on $W$, which is the partial order $\leq$ on $W$ defined by saying that $x\leq y$ if some (equivalently, every) reduced word for $y$ contains a reduced word for $x$ as a (not necessarily contiguous) subword. If $x\leq_R y$ or $x\leq_L y$, then $x\leq y$. 

The poset $(W,\leq_R)$ is a complete meet-semilattice \cite{Bjorner, BjornerBrenti}, meaning that every set $A\subseteq W$ has a unique meet $\bigwedge_R A$. If $A$ has an upper bound in $(W,\leq_R)$, then it has a join, which we denote by $\bigvee_R A$. Similarly, we write $\bigwedge_L A$ and $\bigvee_L A$ for, respectively, the meet of $A$ and the join of $A$ (if it exists) in the left weak order. The right and left weak orders on $W$ are lattices if $W$ is finite. 

A \dfn{right descent} of an element $w\in W$ is a simple generator $s\in S$ such that $\ell(ws)<\ell(w)$; the collection of all right descents of $w$ is the \dfn{right descent set} of $W$, which we denote by $D_R(w)$. Similarly, the \dfn{left descent set} of $w$ is the set $D_L(w)=\{s\in S:\ell(sw)<\ell(w)\}$ of \dfn{left descents} of $w$. It is a basic fact that $D_R(w)=\{s\in S:s\leq_L w\}$ and $D_L(w)=\{s\in S:s\leq_R w\}$. 

Suppose $J\subseteq S$. The \dfn{parabolic subgroup} $W_J$ is the subgroup of $W$ generated by the elements of $J$. For every $w\in W$, the right coset $W_Jw$ has a unique representative $\prescript{J}{}\!w$ of minimal length. Let $w_J=w(\prescript{J}{}\!w)^{-1}$ so that $w_J\in W_J$. The factorization $w=w_J\prescript{J}{}\!w$ is length-additive in the sense that $\ell(w)=\ell(w_J)+\ell(\prescript{J}{}\!w)$. The set $\prescript{J}{}\!W=\{\prescript{J}{}\!w:w\in W\}$ is called a \dfn{parabolic quotient}. The next lemma states that for any $J\subseteq S$, the map $w\mapsto\prescript{J}{}\!w$ is order-preserving with respect to the right weak order. 

\begin{lemma}\label{LemCox3}
Let $(W,S)$ be a Coxeter system, and let $J\subseteq S$. If $y,z\in W$ are such that $y\leq_R z$, then $\prescript{J}{}\!y\leq_R\prescript{J}{}\!z$.
\end{lemma}

\begin{proof}
It suffices to prove this in the case when $z=ys$ for some $s\in S\setminus D_R(y)$. Deodhar's Lemma (see, for example, \cite[Lemma~2.1.2]{Geck}) tells us that either $\prescript{J}{}\!ys\in\prescript{J}{}\!W$ or $\prescript{J}{}\!y s= s'\prescript{J}{}\!y$ for some $s'\in J$. Assume first that $\prescript{J}{}\!ys\in\prescript{J}{}\!W$. Then $z=ys=y_J(\prescript{J}{}\!ys)$, so $z_J=y_J$ and $\prescript{J}{}\!z=\prescript{J}{}\!ys$. Since $s\not\in D_R(y)$ and $\prescript{J}{}\!y\leq_L y$, we must have $s\not\in D_R(\prescript{J}{}\!y)$. Hence, $\prescript{J}{}\!y\leq_R \prescript{J}{}\!ys=\prescript{J}{}\!z$. Now suppose $\prescript{J}{}\!ys\not\in \prescript{J}{}\!W$.  Then $W_J\prescript{J}{}\!ys=W_Js'\prescript{J}{}\!y=W_J\prescript{J}{}\!y$, so $z=y_J(\prescript{J}{}\!ys)\in W_J\prescript{J}{}\!ys=W_J\prescript{J}{}\!y$. It follows that $\prescript{J}{}\!z=\prescript{J}{}\!y$. 
\end{proof}

If $W$ is finite, we write $w_0$ for the longest element of $W$. The \dfn{nerve} of $(W,S)$, denoted $\mathcal N(W,S)$, is the collection of subsets $J\subseteq S$ such that $W_J$ is finite. If $J\in\mathcal N(W,S)$, we write $w_0(J)$ for the longest element of $W_J$. The elements $w_0$ and $w_0(J)$ are involutions. It follows from \cite[Lemma~3.2.4]{BjornerBrenti} that \begin{equation}\label{EqCox1}
\Pop(w)=ww_0(D_R(w))
\end{equation} for every $w\in W$. 
According to \cite[Lemma~3.2.3]{BjornerBrenti}, $W_J$ is finite if and only if $\bigvee_L J$ exists; moreover, if $W_J$ is finite, then $\bigvee_L J=w_0(J)$. Since each $w\in W$ is an upper bound for $D_R(w)$ in the left weak order, the join $\bigvee_L D_R(w)$ must exist and equal $w_0(D_R(w))$. 

\begin{proof}[Proof of Proposition~\ref{PropCox2}]
If $J\in\mathcal N(W,S)$, then $D_R(w_0(J))=J$ since $w_0(J)=\bigvee_LJ$. For every $J\in\mathcal N(W,S)$, it follows from \eqref{EqCox1} that $\Pop(w_0(J))=(w_0(J))^2=e$, so $w_0(J)$ is $1$-pop-stack-sortable. On the other hand, if $v\in W$ is $1$-pop-stack-sortable, then \eqref{EqCox1} implies that $v=w_0(D_R(v))$, and we know that $D_R(v)\in\mathcal N(W,S)$ because $v$ is an upper bound for $D_R(v)$ in the left weak order. Hence, we have a surjective map $\mathcal N(W,S)\to\Pop_W^{-1}(e)$ given by $J\mapsto w_0(J)$. The identity $D_R(w_0(J))=J$ implies that this map is actually a bijection.  
\end{proof}

Fix $w\in W$. If $v\in W$ is such that $D_R(v)=D_R(w)$, then it follows from the above discussion that $w_0(D_R(w))=w_0(D_R(v))\leq_L v$. The right descent set of $w_0(D_R(w))$ is $D_R(w)$, so $w_0(D_R(w))$ is the smallest element in the left weak order that has the same right descent set as $w$. Since the element $w_0(D_R(w))$ is an involution, this explains why, as mentioned in Section~\ref{Subsec:Further}, the Coxeter stack-sorting operator corresponding to the descent congruence on $(W,\leq_L)$ agrees with the Coxeter pop-stack-sorting operator $\Pop_W$.

\subsection{The Pop-Stack-Sorting Map}\label{Subsec:PopMap}

The prototypical example of a Coxeter group is the symmetric group $S_n$. The set of simple generators for $S_n$ is $S=\{s_1,\ldots,s_{n-1}\}$, where $s_i$ is the transposition that swaps $i$ and $i+1$. We often write a permutation $w\in S_n$ as a word $w(1)\cdots w(n)$ in one-line notation. A simple transposition $s_i$ is a right descent of $w$ if and only if $w(i)>w(i+1)$. An \dfn{inversion} of $w$ is a pair $(i,j)$ such that $1\leq i<j\leq n$ and $w(i)>w(j)$. It is well known that the number of inversions of $w$ is $\ell(w)$. The longest element of $S_n$, which we denote by $w_0(S_n)$ when we wish to stress the dependence on $n$, is the decreasing permutation $n(n-1)\cdots 321$. 

The \dfn{direct sum} of permutations $u\in S_m$ and $v\in S_n$ is the permutation $u\oplus v\in S_{m+n}$ defined by \[(u\oplus v)(i)=\begin{cases} u(i) & \mbox{if } 1\leq i\leq m; \\ v(i-m)+m & \mbox{if } m+1\leq i\leq m+n. \end{cases}\] A permutation is called \dfn{layered} if it can be written as $w_0(S_{n_1})\oplus\cdots\oplus w_0(S_{n_k})$, a direct sum of decreasing permutations. Note that layered permutations are involutions. A permutation $w$ is layered if and only if $w=w_0(D_R(w))$. 

A \dfn{descending run} of a permutation $w\in S_n$ is a maximal consecutive decreasing subsequence of $w$. For instance, the descending runs of $42135867$ are $421$, $3$, $5$, $86$, and $7$. The \dfn{pop-stack-sorting map} is the operator on $S_n$ that acts by reversing the descending runs of a permutation while keeping entries in different descending runs in the same relative order. For example, the pop-stack-sorting map sends $42135867$ to $12435687$. Equivalently, the pop-stack-sorting map acts by multiplying a permutation $w$ on the right by the unique layered permutation that has the same descent set as $w$. In other words, it sends $w$ to $ww_0(D_R(w))$. In the above example, we have $w_0(D_R(42135867))=w_0(\{s_1,s_2,s_6\})=32145768$, and we saw that the pop-stack-sorting map sends $42135867$ to the permutation $12435687=42135867\cdot32145768$. Hence, \eqref{EqCox1} tells us that the Coxeter pop-stack-sorting operator $\Pop:S_n\to S_n$ agrees with the pop-stack-sorting map.  

\section{Maximum Forward Orbit Sizes}\label{SecMaximal}

The goal of this section is to prove Theorems~\ref{ThmCox1} and \ref{ThmCox2}. Our first step will be to establish Theorem~\ref{ThmCox1} for finite dihedral groups. 

The dihedral group $I_2(m)$ has presentation $\langle s_1,s_2:s_1^2=s_2^2=(s_1s_2)^m=e\rangle$. For $1\leq k\leq m$, let $\alpha_k$ denote the element $\cdots s_1s_2s_1$ of $I_2(m)$ obtained by multiplying the sequence of $k$ simple generators that alternates between $s_1$ and $s_2$ and ends in $s_1$. Similarly, let $\beta_k=\cdots s_2s_1s_2$ be the product of $k$ simple generators that alternate between $s_1$ and $s_2$ and end in $s_2$. Then $I_2(m)=\{e\}\cup\{\alpha_j:1\leq j\leq m\}\cup\{\beta_k:1\leq k\leq m-1\}$. We have $D_R(e)=\emptyset$ and $D_R(\alpha_m)=\{s_1,s_2\}$. For $1\leq k\leq m-1$, we have $D_R(\alpha_k)=\{s_1\}$ and $D_R(\beta_k)=\{s_2\}$. Thus, by \eqref{EqCox1}, the map $\Pop:I_2(m)\to I_2(m)$ is given explicitly by $\Pop(e)=\Pop(\alpha_m)=\Pop(s_1)=\Pop(s_2)=e$, $\Pop(\alpha_j)=\beta_{j-1}$ for all $2\leq j\leq m-1$, and $\Pop(\beta_k)=\alpha_{k-1}$ for all $2\leq k\leq m-1$. Therefore, the elements of $I_2(m)$ whose forward orbits under $\Pop$ are of maximum size are $\alpha_{m-1}$ and $\beta_{m-1}$. Indeed, $\left|O_{\Pop}(\alpha_{m-1})\right|=\left|O_{\Pop}(\beta_{m-1})\right|=m$. The Coxeter number of $I_2(m)$ is $m$, so this proves Theorem~\ref{ThmCox1} for finite dihedral groups.   

The Coxeter numbers of finite irreducible Coxeter groups are all known \cite[Section~3.18]{Humphreys}; in particular, the only finite irreducible Coxeter groups whose Coxeter numbers are odd are the symmetric groups $S_n$ with $n$ odd and the dihedral groups $I_2(m)$ with $m$ odd. Ungar already proved Theorem~\ref{ThmCox1} for symmetric groups, and we just established this theorem for dihedral groups. Thus, we have the following lemma, which shows that we can focus much of our attention hereafter on finite irreducible Coxeter groups with even Coxeter numbers.  

\begin{lemma}\label{LemCox6}
If $W$ is a finite irreducible Coxeter group whose Coxeter number $h$ is odd, then \[\max\limits_{w\in W}\left|O_{\Pop}(w)\right|=h.\] 
\end{lemma}

The Coxeter diagrams of the finite irreducible Coxeter groups $W$ have been classified (see \cite[Appendix~A1]{BjornerBrenti}); in particular, every such diagram is a tree. This means that it is possible to choose a bipartition $X\sqcup Y$ of the vertex set $S$ of $\Gamma(W)$ (i.e., every edge in $\Gamma(W)$ has one endpoint in $X$ and one endpoint in $Y$). The elements of $X$ all commute with each other, and the elements of $Y$ all commute with each other. The map $W\to W$ given by $w\mapsto w_0ww_0$ is an automorphism of $W$ such that $w_0Sw_0=S$ \cite[Section~2.3]{BjornerBrenti}. In particular, this map induces an automorphism of $\Gamma(W)$. 

\begin{lemma}\label{LemCox1}
Let $W$ be a finite irreducible Coxeter group with an even Coxeter number. If $X\sqcup Y$ is a bipartition of the vertex set $\Gamma(W)$, then $w_0Xw_0=X$. 
\end{lemma}
  
\begin{proof}
Referring to the classification of Coxeter diagrams of finite irreducible Coxeter groups, we immediately find that every automorphism of $\Gamma(W)$ must fix the set $X$ unless $W$ is isomorphic to a dihedral group, a symmetric group $S_n$ with $n$ odd, or the exceptional group $F_4$. Because $W$ has an even Coxeter number, it is not isomorphic to a symmetric group $S_n$ with $n$ odd or a dihedral group $I_2(m)$ with $m$ odd. It is known (one can even check by hand) that if $W$ is isomorphic to $F_4$ or to $I_2(m)$ with $m$ even, then the automorphism given by $w\mapsto w_0ww_0$ is trivial.
\end{proof}

\begin{lemma}\label{LemCox2}
Let $W$ be a finite Coxeter group. Choose $s\in S$, and let $J=S\setminus\{s\}$. We have $D_R(\prescript{J}{}\!w_0)=\{w_0sw_0\}$. 
\end{lemma}

\begin{proof}
Since the map $w\mapsto w_0ww_0$ is an automorphism of $W$ that fixes the set $S$, we have $\ell(w_0ww_0)=\ell(w)$ for all $w\in W$. The element $(w_0)_J=w_0(\prescript{J}{}\!w_0)^{-1}$ is equal to the involution $w_0(J)$ \cite[Section~2.5]{BjornerBrenti}, so $w_0(\prescript{J}{}\!w_0)^{-1}=(\prescript{J}{}\!w_0)w_0$. Therefore, for every $r\in S$, we have \[\ell(\prescript{J}{}\!w_0\cdot w_0rw_0)=\ell(w_0(\prescript{J}{}\!w_0)^{-1}rw_0)=\ell(w_0(r\cdot\prescript{J}{}\!w_0)^{-1}w_0)=\ell((r\cdot\prescript{J}{}\!w_0)^{-1})=\ell(r\cdot\prescript{J}{}\!w_0).\] It follows that $D_R(\prescript{J}{}\!w_0)=w_0D_L(\prescript{J}{}\!w_0)w_0$, so we are left to prove that $D_L(\prescript{J}{}\!w_0)=\{s\}$. An alternative characterization of the parabolic quotient $\prescript{J}{}\!W=\{\prescript{J}{}\!w:w\in W\}$ is that it is the set of elements of $W$ whose left descent sets are contained in the set $S\setminus J=\{s\}$ (see \cite[Section~2.4]{BjornerBrenti}). Since $\prescript{J}{}\!w_0\in\prescript{J}{}\!W$ is not the identity element $e$, its left descent set must be $\{s\}$. 
\end{proof}

The next proposition is the key to proving both Theorems~\ref{ThmCox1} and \ref{ThmCox2}. 

\begin{proposition}\label{PropCox1}
Let $W$ be a finite irreducible Coxeter group whose Coxeter number $h$ is even. Choose $s\in S$, and let $J=S\setminus \{s\}$. We have $\Pop^{h-1}(\prescript{J}{}\!w_0)=e$ and $\Pop^{h-2}(\prescript{J}{}\!w_0)\neq e$. Furthermore, for every $t\geq 0$, the right descents of $\Pop^t(\prescript{J}{}\!w_0)$ all commute with each other.  
\end{proposition}

\begin{proof}
Let us choose a bipartition $X\sqcup Y$ of the vertex set $S$ of $\Gamma(W)$. Without loss of generality, assume $s\in X$. All of the elements of $X$ commute with each other, so it makes sense to define the product $c_X=\prod_{r\in X}r$. Similarly, we can define $c_Y=\prod_{r\in Y}r$. The element $c=c_Xc_Y$ is a Coxeter element of $W$. Let us fixed a reduced word $r_1\cdots r_{|S|}$ of $c_Xc_Y$ such that $r_1\cdots r_{|X|}$ is a reduced word for $c_X$ and $r_{|X|+1}\cdots r_{|S|}$ is a reduced word for $c_Y$. Without loss of generality, we may assume $r_1=s$. According to \cite[Chapter~V, Section~6.2, Proposition~2]{Bourbaki}, we have $w_0=(c_Xc_Y)^{h/2}$. It is well known \cite[Section~3.18]{Humphreys} that $\ell(w_0)=|S|\cdot h/2$, so concatenating $r_1\cdots r_{|S|}$ with itself $h/2$ times produces a reduced word $s_1\cdots s_{\ell(w_0)}$ for $w_0$. Note that $s_1=r_1=s$. Let us write the word $s_1\cdots s_{\ell(w_0)}$ as the concatenation $u_1\cdots u_h$, where $u_i=r_1\cdots r_{|X|}$ when $i$ is odd and $u_i=r_{|X|+1}\cdots r_{|S|}$ when $i$ is even. We think of the letters $s_1,\ldots,s_{\ell(w_0)}$ as distinct entities, and we think of each $s_j$ as belonging to exactly one of the words $u_i$. For example, $s_1$ belongs to $u_1$, while $s_{\ell(w_0)}$ belongs $u_h$. 

Let $k=\ell(w_0)-\ell(\prescript{J}{}\!w_0)$. Given indices $1\leq p_1<\cdots<p_k\leq\ell(w_0)$ and $j\in[k]$, we can consider the word $s_1\cdots\widehat s_{p_1}\cdots\widehat s_{p_j}\cdots s_{\ell(w_0)}$ obtained by removing the letters in positions $p_1,\ldots,p_j$ from the word $s_1\cdots s_{\ell(w_0)}$. Let $z^{(j)}=s_1\cdots\widehat s_{p_1}\cdots\widehat s_{p_j}\cdots s_{\ell(w_0)}$ be the element of $W$ represented by this word. The proof of \cite[Theorem~2.5.5]{BjornerBrenti} shows that it is possible to choose the indices $p_1<\cdots<p_k$ so that $\ell(z^{(j)})=\ell(w_0)-j$ for all $j\in[k]$ and so that $z^{(k)}=\prescript{J}{}\!w_0$. Let us assume that we have made such a choice. Let $v_i$ be the word obtained from $u_i$ by removing the letters in the sequence $s_{p_1},\ldots,s_{p_k}$ that belong to $u_i$. Thus, $v_1\cdots v_h$ is the reduced word $s_1\cdots\widehat s_{p_1}\cdots\widehat s_{p_k}\cdots s_{\ell(w_0)}$ for $\prescript{J}{}\!w_0$. 

We claim that none of the words $v_1,\ldots,v_{h-1}$ are empty. To prove this, we first assume, by way of contradiction, that $v_1$ is empty. This implies that $p_1=1$, so the reduced word $v_1\cdots v_h$ for $\prescript{J}{}\!w_0$ is contained in the reduced word $s_2\cdots s_{\ell(w_0)}$ for $sw_0$. Hence, $\prescript{J}{}\!w_0\leq sw_0$ in the strong Bruhat order. The map $W\to W$ given by $w\mapsto ww_0$ is an antiautomorphism of the strong Bruhat order \cite[Proposition~2.3.4]{BjornerBrenti}, so $s\leq (\prescript{J}{}\!w_0)w_0=((w_0)_J)^{-1}$. Since $((w_0)_J)^{-1}\in W_J$, this forces $s\in J$, which contradicts the fact that $J=S\setminus\{s\}$. Consequently, $v_1$ is not the empty word. 

Now suppose there is some $i\in\{2,\ldots,h-1\}$ such that $v_i$ is the empty word; we may assume that this $i$ is chosen minimally. Let $j$ be the index such that $s_{p_j}$ belongs to $u_i$ and is the last letter in $u_i$. Then $z^{(j-1)}=v_1\cdots v_{i-1}s_{p_j}u_{i+1}\cdots u_h$ and $z^{(j)}=v_1\cdots v_{i-1}u_{i+1}\cdots u_h$. By the minimality of $i$, the word $v_{i-1}$ is nonempty. Let $s'$ be the last letter of $v_{i-1}$, and write $v_{i-1}=\widetilde v_{i-1}s'$. Because $i-1$ and $i+1$ have the same parity, $s'$ is an element of $S$ that appears in $u_{i+1}$. Furthermore, the elements of $S$ appearing in $u_{i+1}$ all commute with each other, so there is a reduced word $\widetilde u_{i+1}$ such that $u_{i+1}=s'\widetilde u_{i+1}$ (as elements of $W$). This means that $z^{(j)}=v_1\cdots \widetilde v_{i-1}s's'\widetilde u_{i+1}\cdots u_h=v_1\cdots \widetilde v_{i-1}\widetilde u_{i+1}\cdots u_h$, so $\ell(z^{(j)})\leq\ell(z^{(j-1)})-3$. This contradicts the fact that $\ell(z^{(j)})=\ell(w_0)-j=\ell(z^{(j-1)})-1$. Hence, we have proven the claim that $v_1,\ldots,v_{h-1}$ are nonempty.   

Now recall that $v_1\cdots v_h=\prescript{J}{}\!w_0$. Every letter in $v_h$ is also a letter in $u_h$, and the letters in $u_h$ are the elements of $Y$ because $h$ is even. It follows that if $v_h$ is nonempty, then its last letter is in $D_R(\prescript{J}{}\!w_0)\cap Y$. Lemma~\ref{LemCox2} tells us that $D_R(\prescript{J}{}\!w_0)=\{w_0sw_0\}$. Since $s\in X$, it follows from Lemma~\ref{LemCox1} that $w_0sw_0\in X$. Hence, $D_R(\prescript{J}{}\!w_0)\cap Y=\emptyset$. This demonstrates that $v_h$ is the empty word, so $\prescript{J}{}\!w_0=v_1\cdots v_{h-1}$. 

Choose some $m\in[h-1]$. Because the letters in $v_m$ all commute with each other, they must all be right descents of $v_1\cdots v_m$. We wish to show that the set of letters appearing in $v_m$ is actually equal to $D_R(v_1\cdots v_m)$. Let $x_m$ be a word obtained by multiplying the letters that appear in $u_m$ but not $v_m$ in some order. Then $u_m=v_mx_m$ (as elements of $W$), so $v_1\cdots v_mx_mu_{m+1}\cdots u_h$ is a reduced word for some $z^{(j)}$. Consequently, the word $v_1\cdots v_mx_m$ is reduced. All of the letters in $x_m$ commute with each other, so none of them can be right descents of $v_1\cdots v_m$. There is an index $j'$ such that $v_1\cdots v_mu_{m+1}\cdots u_h$ is a reduced word for $z^{(j')}$. It follows that the word $v_1\cdots v_mu_{m+1}$ is reduced. All of the letters in $u_{m+1}$ commute with each other, so none of them can be right descents of $v_1\cdots v_m$. Since every element of $S$ appears in exactly one of the words $v_m, x_m, u_{m+1}$, this proves that $D_R(v_1\cdots v_m)$ is the set of letters appearing in $v_m$. Because the letters in $v_m$ all commute with each other, the element of $W$ represented by $v_m$ is $w_0(D_R(v_1\cdots v_m))$. Appealing to \eqref{EqCox1}, we find that $\Pop(v_1\cdots v_m)=v_1\cdots v_mw_0(D_R(v_1\cdots v_m))=v_1\cdots v_{m-1}(w_0(D_R(v_1\cdots v_m)))^2=v_1\cdots v_{m-1}$.  

It now follows by induction on $t$ that $\Pop^t(\prescript{J}{}\!w_0)=v_1\cdots v_{h-t-1}$ for every $0\leq t\leq h-1$. In particular, $\Pop^{h-1}(\prescript{J}{}\!w_0)=e$ and $\Pop^{h-2}(\prescript{J}{}\!w_0)=v_1\neq e$. Furthermore, the right descents of $\Pop^t(\prescript{J}{}\!w_0)$ are the letters appearing in $v_{h-t-1}$, which all commute with each other. If $t\geq h$, then it is vacuously true that the right descents of $\Pop^t(\prescript{J}{}\!w_0)=e$ commute with each other.  
\end{proof}

Proposition~\ref{PropCox1} tells us already that if $W$ is finite and irreducible with an even Coxeter number $h$, then $\sup\limits_{w\in W}\left|O_{\Pop}(w)\right|\geq h$. To prove the reverse inequality and Theorem~\ref{ThmCox2}, we need a couple more preparatory lemmas. Recall that $\leq$ denotes the strong Bruhat order. 

\begin{lemma}\label{LemCox4}
Let $W$ be a finite Coxeter group. Let $x,y\in W$, and assume that all of the right descents of $y$ commute with each other. If $x\leq y$, then $\Pop(x)\leq\Pop(y)$.  
\end{lemma}

\begin{proof}
Let $s_1,\ldots,s_r$ be the right descents of $y$. Because $s_1,\ldots,s_r$ all commute with each other, we have $w_0(D_R(y))=s_1\cdots s_r$. Let $s_1'\cdots s_q'$ be a reduced word for $\Pop(y)$. Since $\Pop(y)\leq_R y$ by definition, it follows from \eqref{EqCox1} and the fact that $w_0(D_R(y))$ is an involution that $s_1'\cdots s_q's_1\cdots s_r$ is a reduced word for $y$. Assume $x\leq y$. This implies that the reduced word $s_1'\cdots s_q's_1\cdots s_r$ contains a reduced word $s_{i_1}'\cdots s_{i_a}'s_{j_1}\cdots s_{j_b}$ for $x$. Here, $s_{i_1}'\cdots s_{i_a}'$ is a reduced word for some element $z$ with $z\leq \Pop(y)$. Note that $\ell(z)=\ell(x)-b$. Since $s_{j_1},\ldots, s_{j_b}$ all commute with each other, they are all in $D_R(x)$. Thus, $s_{j_1}\cdots s_{j_b}$ is in the parabolic subgroup $W_{D_R(x)}$. Now, $w_0(D_R(x))$ is the unique maximal element in the left weak order on $W_{D_R(x)}$, so $s_{j_1}\cdots s_{j_b}\leq_L w_0(D_R(x))$. This means that there exists $v\in W$ such that $w_0(D_R(x))=vs_{j_1}\cdots s_{j_b}$ and $\ell(w_0(D_R(x)))=\ell(v)+b$. Using \eqref{EqCox1} and the fact that $w_0(D_R(x))$ is an involution, we deduce that \[zs_{j_1}\cdots s_{j_b}=x=\Pop(x)w_0(D_R(x))=\Pop(x)vs_{j_1}\cdots s_{j_b},\] so $z=\Pop(x)v$. Since $\Pop(x)\leq_R x$, the factorization $x=\Pop(x)w_0(D_R(x))$ is length-additive, meaning $\ell(x)=\ell(\Pop(x))+\ell(w_0(D_R(x)))$. Thus, $\ell(z)=\ell(x)-b=\ell(\Pop(x))+\ell(w_0(D_R(x)))-b=\ell(\Pop(x))+\ell(v)$. This shows that $\Pop(x)\leq_R z$, so $\Pop(x)\leq z\leq \Pop(y)$.   
\end{proof}

\begin{lemma}\label{LemCox5}
Let $(W,S)$ be a Coxeter system, and let $J\subseteq S$. If $f:W\to W$ is compulsive, then $\prescript{J}{}\!(f(w))\leq_R\Pop(\prescript{J}{}\!w)$ for every $w\in W$.  
\end{lemma}

\begin{proof}
By definition, $\Pop(\prescript{J}{}\!w)=\bigwedge_R\{\prescript{J}{}\!wx:x\in D_R(\prescript{J}{}\!w)\cup\{e\}\}$. Therefore, it suffices to prove that $\prescript{J}{}\!(f(w))\leq_R\prescript{J}{}\!wx$ for all $x\in D_R(\prescript{J}{}\!w)\cup\{e\}$. Fix $x\in D_R(\prescript{J}{}\!w)\cup\{e\}$. According to \cite[Proposition~2.5]{Stembridge}, the parabolic quotient $\prescript{J}{}\!W$ is an order ideal in the right weak order on $W$. Since $\prescript{J}{}\!wx\leq_R\prescript{J}{}\!w$, it follows that $\prescript{J}{}\!wx\in\prescript{J}{}\!W$. Because $wx=w_J\cdot\prescript{J}{}\!wx$, we must have $\prescript{J}{}\!(wx)=\prescript{J}{}\!wx$. Now, $x\in D_R(\prescript{J}{}\!w)\cup\{e\}\subseteq D_R(w)\cup\{e\}$, so it follows from the hypothesis that $f$ is compulsive that $f(w)\leq_R wx$. Invoking Lemma~\ref{LemCox3}, we see that $\prescript{J}{}\!(f(w))\leq_R\prescript{J}{}\!(wx)=\prescript{J}{}\!wx$, as desired.  
\end{proof}

We can now complete the proof of Theorem~\ref{ThmCox2}.

\begin{proof}[Proof of Theorem~\ref{ThmCox2}]
Let $W$ be an irreducible Coxeter group with Coxeter number $h$, and let $f:W\to W$ be compulsive. If $W$ is infinite, then $h=\infty$, so the result is trivial. Therefore, we may assume $W$ is finite. Fix $w\in W$; our goal is to prove that $f^{h-1}(w)=e$. It suffices to prove that $f^{h-1}(w)\in W_J$ for every set $J\subseteq S$ such that $|J|=|S|-1$. Fix such a set $J$. Note that the desired containment $f^{h-1}(w)\in W_J$ is equivalent to the identity $\prescript{J}{}\!(f^{h-1}(w))=e$. 

Because $w\leq_R w_0$, we have $\prescript{J}{}\!w\leq_R\prescript{J}{}\!w_0$ by Lemma~\ref{LemCox3}. Therefore, $\prescript{J}{}\!w\leq \prescript{J}{}\!w_0$. By induction on $t$, we will prove that $\prescript{J}{}\!(f^t(w))\leq \Pop^t(\prescript{J}{}\!w_0)$ for all $t\geq 0$. We have just established the base case $t=0$, so let us assume that $t\geq 1$ and that we have already proven the inequality $\prescript{J}{}\!(f^{t-1}(w))\leq \Pop^{t-1}(\prescript{J}{}\!w_0)$. Proposition~\ref{PropCox1} tells us that the right descents of $\Pop^{t-1}(\prescript{J}{}\!w_0)$ all commute with each other, so we can use Lemma~\ref{LemCox4} with $x=\prescript{J}{}\!(f^{t-1}(w))$ and $y=\Pop^{t-1}(\prescript{J}{}\!w_0)$ to see that $\Pop(\prescript{J}{}\!(f^{t-1}(w)))\leq\Pop^t(\prescript{J}{}\!w_0)$. Lemma~\ref{LemCox5} tells us that $\prescript{J}{}\!(f^t(w))\leq_R\Pop(\prescript{J}{}\!(f^{t-1}(w)))$, so $\prescript{J}{}\!(f^t(w))\leq\Pop(\prescript{J}{}\!(f^{t-1}(w)))\leq\Pop^t(\prescript{J}{}\!w_0)$. This completes the induction step. Now set $t=h-1$ to find that $\prescript{J}{}\!(f^{h-1}(w))\leq\Pop^{h-1}(\prescript{J}{}\!w_0)$. It follows from Lemma~\ref{LemCox6} (if $h$ is odd) and Proposition~\ref{PropCox1} (if $h$ is even) that $\Pop^{h-1}(\prescript{J}{}\!w_0)=e$. We deduce that $\prescript{J}{}\!(f^{h-1}(w))=e$, as desired.  
\end{proof}

Notice how the proof of Theorem~\ref{ThmCox2} and the arguments leading up to it have utilized the right weak order in tandem with the strong Bruhat order in a subtle manner. 

Let us now wrap up the proof of Theorem~\ref{ThmCox1}.

\begin{proof}[Proof of Theorem~\ref{ThmCox1}]
Let $W$ be an irreducible Coxeter group with Coxeter number $h$. We first assume $W$ is finite. If $h$ is odd, then we are done by Lemma~\ref{LemCox6}, so assume $h$ is even. Choose $s\in S$, and let $J=S\setminus\{s\}$. We saw in Proposition~\ref{PropCox1} that $\Pop^{h-2}(\prescript{J}{}\!w_0)\neq e$, so $\max\limits_{w\in W}\left|O_{\Pop}(w)\right|\geq\left|O_{\Pop}(\prescript{J}{}\!w_0)\right|\geq h$. On the other hand, the inequality $\max\limits_{w\in W}\left|O_{\Pop}(w)\right|\leq h$ follows immediately from Theorem~\ref{ThmCox2} since $\Pop$ is compulsive. 

Now assume $W$ is infinite so that $h=\infty$. Let us first suppose that $S$ is finite. Let $\mathcal N(W,S)=\{J\subseteq S:|W_J|<\infty\}$ be the nerve of $(W,S)$. Let $K=\max\limits_{J\in\mathcal N(W,S)}\ell(w_0(J))$. Choose $v\in W$. Since $v$ is an upper bound for $D_R(v)$ in the left weak order on $W$, it follows from \cite[Lemma~2.3.2]{BjornerBrenti} that $D_R(v)\in\mathcal N(W,S)$. Using \eqref{EqCox1} and the fact that $\Pop(v)\leq_R v$, we obtain \[\ell(\Pop(v))=\ell(v)-\ell((\Pop(v))^{-1}v)=\ell(v)-\ell((w_0(D_R(v)))^{-1})=\ell(v)-\ell(w_0(D_R(v)))\geq \ell(v)-K.\] As this is true for all $v\in W$, we find that $\left|O_{\Pop}(w)\right|\geq\ell(w)/K+1$ for all $w\in W$. An infinite irreducible Coxeter group contains arbitrarily long elements, so $\sup\limits_{w\in W}\left|O_{\Pop}(w)\right|=\infty$. 

Finally, suppose $S$ is infinite. Observe that if $J\subseteq S$, then $\Pop_{W_{J}}(w)=\Pop_W(w)$ for every $w\in W_{J}$. If there exists a finite set $S'\subseteq S$ such that $W_{S'}$ is infinite and irreducible, then it follows from the preceding paragraph that $\sup\limits_{w\in W}\left|O_{\Pop_W}(w)\right|\geq\sup\limits_{w\in W_{S'}}\left|O_{\Pop_{W_{S'}}}(w)\right|=\infty$. Now suppose no such set $S'$ exists. Let $S_1\subsetneq S_2\subsetneq \cdots$ be an infinite (strictly increasing) chain of finite subsets of $S$ such that for each $i\geq 1$, the induced subgraph of $\Gamma(W)$ on the vertex set $S_i$ is connected. The parabolic subgroups $W_{S_i}$ are irreducible and finite. Let $h_i$ be the Coxeter number of $W_{S_i}$. For every $i\geq 1$, we have $\sup\limits_{w\in W}\left|O_{\Pop_W}(w)\right|\geq\sup\limits_{w\in W_{S_i}}\left|O_{\Pop_{W_{S_i}}}(w)\right|=h_i$. 
The classification of Coxeter numbers \cite[Section~3.18]{Humphreys} implies that $h_i\to\infty$ as $i\to\infty$, so $\sup\limits_{w\in W}\left|O_{\Pop_W}(w)\right|=\infty$.  
\end{proof}

\section{$t$-Pop-Stack-Sortable Elements in Type $B$}\label{Sec:TypeB}

We begin this section by recalling some basic facts about Coxeter groups of type $B$. The map $w\mapsto w_0ww_0$ is an automorphism of $S_{2n}$; the permutations fixed by this automorphism form a subgroup $B_n$ of $S_{2n}$ called the $n^\text{th}$ \dfn{hyperoctahedral group}. The map $w\mapsto w_0ww_0$ is also a lattice automorphism of the right weak order on $S_{2n}$, so $(B_n,\leq_R)$ is a sublattice of $(S_{2n},\leq_R)$. Let $s_i$ denote the simple transposition in $S_{2n}$ that swaps $i$ and $i+1$. The group $B_n$ is a Coxeter group whose simple generators are the elements $s_1^B,\ldots,s_n^B$ given by $s_i^B=s_is_{2n-i}$ for $i\in[n-1]$ and $s_n^B=s_n$. Let us use $\wedge_R^A$ (respectively, $\wedge_R^B$) and $\bigwedge_R^A$ (respectively, $\bigwedge_R^B$) to denote meets in the right weak order on $S_{2n}$ (respectively, $B_n$). 

Fix $w\in B_n$, and let $\mathcal D$ be the set of indices $i\in[n]$ such that $s_i^B$ is a right descent of $w$ in $B_n$. The right descent set of $w$ as an element of $S_{2n}$ is $\{s_i:i\in\mathcal D\}\cup\{s_{2n-i}:i\in\mathcal D\}$. If $i\in\mathcal D\cap[n-1]$, then $ws_i\wedge_R^A ws_{2n-i}=ws_is_{2n-i}=ws_i^B$ because $s_i$ and $s_{2n-i}$ commute. On the other hand, if $n\in\mathcal D$, then $ws_n\wedge_R^A ws_{2n-n}=ws_n=ws_n^B$. It follows that \[\Pop_{S_{2n}}(w)=\bigwedge\nolimits_R^A(\{ws_i:i\in\mathcal D\}\cup\{ws_{2n-i}:i\in\mathcal D\}\cup\{w\})=\bigwedge\nolimits_R^A(\{ws_i\wedge_R^Aws_{2n-i}:i\in\mathcal D\}\cup\{w\})\] \[=\bigwedge\nolimits_R^A(\{ws_i^B:i\in\mathcal D\}\cup\{w\}).\] Since $(B_n,\leq_R)$ is a sublattice of $(S_{2n},\leq_R)$, this shows that $\Pop_{S_{2n}}(w)=\Pop_{B_n}(w)$. 

For $w\in S_{2n}$, the one-line notation of $w_0ww_0$ is \[(2n+1-w(2n))(2n+1-w(2n-1))\cdots(2n+1-w(1)).\] More geometrically, one can consider the \dfn{plot} of a permutation $v\in S_{2n}$, which is the set of points $(i,v(i))\in\mathbb R^2$ for $i\in[2n]$. The plot of $w_0ww_0$ is obtained by rotating the plot of $w$ by $180^\circ$ about the point $(\frac{n+1}{2},\frac{n+1}{2})$. The previous paragraph shows that $\Pop_{B_n}$ is the restriction of $\Pop_{S_{2n}}$ to $B_n$. The main conclusion we wish to draw from this is that the $t$-pop-stack-sortable elements of $B_n$ are precisely the $t$-pop-stack-sortable elements of $S_{2n}$ that are fixed by the automorphism $w\mapsto w_0ww_0$. In other words, they are the $t$-pop-stack-sortable permutations in $S_{2n}$ whose plots are invariant under $180^\circ$ rotation. 

\subsection{Enumerating $2$-Pop-Stack-Sortable Elements in Type $B$} 
In order to enumerate the $2$-pop-stack-sortable permutations in $B_{n}$, we first state a simple characterization of $2$-pop-stack-sortable permutations in $S_n$ from \cite{Pudwell}; this characterization also follows easily from the description of the pop-stack-sorting map given in Section~\ref{Subsec:PopMap}. 

\begin{lemma}[\cite{Pudwell}]\label{LemCox7}
Suppose $w\in S_n$ is a permutation whose descending runs, read from left to right, are $\delta_1,\ldots,\delta_r$. Then $w$ is $2$-pop-stack-sortable if and only if for every $j\in[r-1]$, the largest (equivalently, the first) entry in $\delta_j$ is at most $1$ more than the smallest (equivalently, the last) entry in $\delta_{j+1}$. 
\end{lemma}

For example, $62135847$ is not $2$-pop-stack-sortable because the largest entry in the first descending run is $6$, the smallest entry in the second descending run is $3$, and $6>3+1$. 

In what follows, we will find it convenient to view permutations more generally as orderings of arbitrary finite subsets of $\mathbb Z$. For example, $3617$ is the one-line notation of a permutation of size $4$. Our convention is that $S_n$ is the set of permutations of $[n]$. The \dfn{standardization} of a permutation $w$ of size $n$ is the permutation in $S_n$ obtained by replacing the $i^\text{th}$-smallest entry in $w$ with $i$ for all $i\in[n]$. For example, the standardization of $3617$ is $2314$.

For $n\geq 1$, let $L(n,k)$ be the set of $2$-pop-stack-sortable permutations in $B_{n-1}$ that have either $2k$ or $2k+1$ descending runs. Let $M(n,k)$ be the set of permutations $w\in L(n,k)$ such that the last descending run of $w$ contains exactly $1$ entry. We make the convention $L(1,0)=M(1,0)=\{\varepsilon\}$, where $\varepsilon$ is the empty permutation. Furthermore, $L(1,k)=M(1,k)=\emptyset$ for $k\neq 0$. The only permutation in $B_{n-1}$ that has $0$ or $1$ descending runs is the decreasing permutation $w_0(S_{2n-2})=(2n-2)(2n-3)\cdots 321$, so $L(n,0)=\{w_0(S_{2n-2})\}$. All of the entries in a decreasing permutation are in the same descending run, so $M(n,0)=\emptyset$ whenever $n\geq 2$.  

\begin{proposition}\label{PropCox3}
Preserve the notation from above. If $n\geq 2$ and $1\leq k\leq n-1$, then \[|L(n,k)|=2\sum_{i=1}^{n-1}|L(i,k-1)|-|M(n-1,k-1)|\quad\text{and}\quad |M(n,k)|=2|L(n-1,k-1)|-|M(n-1,k-1)|.\]
\end{proposition}

\begin{proof}
This follows from Lemma~\ref{LemCox7} and the fact that a permutation is in a hyperoctahedral group if and only if its plot is invariant under $180^\circ$ rotation. Indeed, suppose we are given $w\in L(n,k)$. If we remove the first and last descending runs and standardize the resulting permutation, we obtain a permutation $\alpha(w)\in\bigcup_{i=1}^{n-1}L(i,k-1)$. Hence, we have a map $\alpha:L(n,k)\to \bigcup_{i=1}^{n-1}L(i,k-1)$. Suppose $v\in L(i,k-1)$ for some $1\leq i\leq n-1$. If $v\in M(n-1,k-1)$, then, upon inspecting the characterization of $2$-pop-stack-sortable permutations in Lemma~\ref{LemCox7}, we find that $|\alpha^{-1}(v)|=1$. Indeed, the only element of $\alpha^{-1}(v)$ is (in the notation of Section~\ref{Subsec:PopMap}) $1\oplus v\oplus 1$. Now suppose $v\not\in M(n-1,k-1)$. In this case, there are exactly $2$ elements of $\alpha^{-1}(v)$. The first is $w_0(S_{n-i})\oplus v\oplus w_0(S_{n-i})$ (recall from Section~\ref{Subsec:PopMap} that $w_0(S_{n-i})$ is the decreasing permutation of size $n-i$). The second is the permutation whose plot is obtained from that of $w_0(S_{n-i})\oplus v\oplus w_0(S_{n-i})$ by sliding the rightmost point down so that it is immediately below the highest point in the second-to-last descending run and sliding the leftmost point up so that it is immediately above the lowest point in the second descending run. For example, suppose $n=7$, $i=4$, and $v=264315$. The first element of $\alpha^{-1}(v)$ is $w_0(S_3)\oplus v\oplus w_0(S_3)=3\,2\,1\,5\,9\,7\,6\,4\,8\,12\,11\,10$. The other element of $\alpha^{-1}(v)$ is $5\,2\,1\,4\,10\,7\,6\,3\,9\,12\,11\,8$. The latter permutation was obtained from $3\,2\,1\,5\,9\,7\,6\,4\,8\,12\,11\,10$ by decreasing the last entry so that it is $1$ less than the largest entry in the second-to last descending run (i.e., $8=9-1$) and increasing the first entry so that it is $1$ more than the smallest entry in the second descending run (i.e., $5=4+1$). It follows that $|L(n,k)|$ is equal to \[|M(n-1,k-1)|+2\left|\bigcup_{i=1}^{n-1}L(i,k-1)\setminus M(n-1,k-1)\right|=2\sum_{i=1}^{n-1}|L(i,k-1)|-|M(n-1,k-1)|.\] 

Note that if $w\in M(n,k)$, then $\alpha(w)\in L(n-1,k-1)$. Let $\beta:M(n,k)\to L(n-1,k-1)$ be the restriction of $\alpha$ to $M(n,k)$. If $v\in M(n-1,k-1)$, then we saw above that the unique element of $\alpha^{-1}(v)$ is $1\oplus v\oplus 1$, which is in $M(n,k)$. Thus, $|\beta^{-1}(v)|=1$ in this case. If $v\in L(n-1,k-1)\setminus M(n-1,k-1)$, then we saw above how to construct the two elements of $\alpha^{-1}(v)$. Both of these elements are actually in $M(n,k)$, so $|\beta^{-1}(v)|=2$ in this case. This shows that $|M(n,k)|$ is equal to \[2|L(n-1,k-1)\setminus M(n-1,k-1)|+|M(n-1,k-1)|=2|L(n-1,k-1)|-|M(n-1,k-1)|. \qedhere\]
\end{proof}

\begin{proof}[Proof of Theorem~\ref{ThmCox3}]
In \cite{Pudwell}, Pudwell and Smith define $a(n,k)$ to be the number of $2$-pop-stack-sortable permutations in $S_n$ that have exactly $k+1$ descending runs. They also define $b(n,k)$ to be the number of $2$-pop-stack-sortable permutations $w\in S_n$ with exactly $k+1$ descending runs such that the last descending run of $w$ has exactly one entry. They prove (see \cite[Proposition~1]{Pudwell}) that $a(n,k)$ and $b(n,k)$ satisfy the same recurrence relation that we found for $|L(n,k)|$ and $|M(n,k)|$ in Proposition~\ref{PropCox3}, even with the same initial conditions. It follows that $|L(n,k)|=a(n,k)$ and $|M(n,k)|=b(n,k)$. This proves the first statement in Theorem~\ref{ThmCox3} and shows that \[\sum_{n\geq 1}\left|\Pop_{B_n}^{-2}(e)\right|z^n=\sum_{n\geq 1}\left|\Pop_{S_{n+1}}^{-2}(e)\right|z^n.\] Pudwell and Smith also found that the generating function $\sum_{n\geq 0}\left|\Pop_{S_n}^{-2}(e)\right|z^n$ that counts $2$-pop-stack-sortable permutations is equal to $\dfrac{1-z-z^2-z^3}{1-2z-z^2-2z^3}$ (see \cite[Corollary~1]{Pudwell}). Consequently, \[\sum_{n\geq 1}\left|\Pop_{S_{n+1}}^{-2}(e)\right|z^n=\frac{1}{z}\left(\frac{1-z-z^2-z^3}{1-2z-z^2-2z^3}-1\right)-1=\frac{2z(1+z+z^2)}{1-2z-z^2-2z^3}.\qedhere\]
\end{proof}

\subsection{Generating Functions for $t$-Pop-Stack-Sortable Elements in Type $B$}

We now proceed to prove Theorem~\ref{ThmCox5}, which states that for each fixed $t\geq 0$, the generating function that counts $t$-pop-stack-sortable permutations in hyperoctahedral groups is rational. Our proof makes heavy use of the tools that Claesson and Gu{\dh}mundsson developed in \cite{ClaessonPop}. 

Fix an integer $t\geq 1$. Consider a $t$-pop-stack-sortable permutation $w\in S_n$. Let us write out the one-line notations of $w,\Pop(w),\Pop^2(w),\ldots,\Pop^{t-1}(w)$, with $\Pop^k(w)$ directly below $\Pop^{k-1}(w)$ for each $k\in[t-1]$. Draw boxes around the descending runs of each of these permutations. The resulting array of numbers and boxes is called the \dfn{sorting trace} of $w$. If we delete the numbers in this sorting trace, we obtain an array of boxes called the \dfn{sorting plan} of $w$. We call $n$ and $t$ the \dfn{length} and \dfn{order}, respectively, of the sorting trace and the sorting plan. For a concrete example, see Figure~\ref{FigCox3}, which shows the sorting trace and sorting plan of order $4$ of the permutation $6524713$. Let $\SP_n(t)$ denote the set of all sorting plans of length $n$ and order $t$ (i.e., the set of all sorting plans of $t$-pop-stack-sortable permutations in $S_n$). Let $\SP(t)=\bigcup_{n\geq 1}\SP_n(t)$. Each box in the sorting trace or sorting plan of $w$ is called a \dfn{block}; the \dfn{length} of a block is the number of entries that it contains in the sorting trace. 

\begin{figure}[ht]
  \begin{center}{\includegraphics[height=1.811cm]{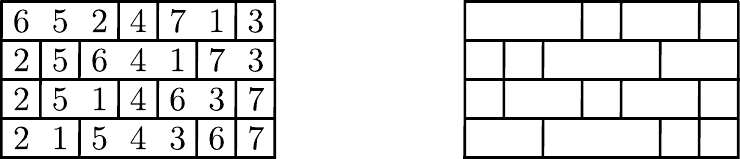}}
  \end{center}
  \caption{The sorting trace (left) and sorting plan (right) of order $4$ of $6524713$.}\label{FigCox3}
\end{figure}

Notice that the sorting trace of $w$ is completely determined by the sorting plan of $w$. Indeed, for each $k\in[t]$, the permutation $\Pop^k(w)$ is obtained by reversing the descending runs of $\Pop^{k-1}(w)$. The lengths of the descending runs of $\Pop^{k-1}(w)$ are precisely the lengths of the blocks in the $k^\text{th}$ row of the sorting plan of $w$. Thus, $\Pop^{k-1}(w)$ can be obtained by inserting $\Pop^k(w)$ into the $k^\text{th}$ row of the sorting plan of $w$ and then reversing the entries within each block. Since $\Pop^t(w)=e$, this shows that we can determine $w$, and, hence, the sorting trace of $w$, from the sorting plan of $w$. It follows that $t$-pop-stack-sortable permutations in $S_n$ are in bijection with sorting plans of length $n$ and order $t$.

Define a \dfn{bar code} of order $t$ to be a sequence of $t$ vertical bars and blank spaces, each of height $1$, arranged vertically from top to bottom. If we associate each vertical bar with the digit $0$ and associate each blank space with the digit $1$, then a bar code corresponds to an element of $\{0,1\}^t$. We associate each element $(x_1,\ldots,x_t)$ of $\{0,1\}^t$ with the integer $x_t+2x_{t-1}+2^2x_{t-2}+\cdots+2^{t-1}x_1\in\Sigma_t$, where $\Sigma_t=\{0,\ldots,2^t-1\}$. Hence, we have a bijective correspondence between bar codes of order $t$ and elements of $\Sigma_t$. For example, with $t=4$, the bar codes in Figure~\ref{FigCox4} correspond, from left to right, to the vectors $(0,0,0,0)$, $(0,0,1,0)$, and $(1,1,0,1)$. These, in turn, correspond to the numbers $0$, $2$, and $13$ in $\Sigma_4$.  

\begin{figure}[ht]
  \begin{center}{\includegraphics[height=1.811cm]{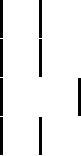}}
  \end{center}
  \caption{Three bar codes of order $4$.}\label{FigCox4}
\end{figure}

A \dfn{segment} of length $n$ and order $t$ is a sequence of $n+1$ bar codes of order $t$ that are arranged from left to right and separated by spaces of width $1$, along with $t+1$ separating horizontal lines of length $n$. For instance, the left side of Figure~\ref{FigCox5} depicts a segment of length $4$ and order $4$. Being (essentially) a sequence of bar codes, a segment $\sigma$ corresponds to a word $\psi(\sigma)$ over the alphabet $\Sigma_t$. For example, if $\sigma$ is the segment on the left of Figure~\ref{FigCox5}, then $\psi(\sigma)$ is the word $9\, 10\, 5\, 5\, 10$. Notice that $\psi$ is an injection from the set of segments of order $t$ into the set of words over $\Sigma_t$. Every sorting plan is a segment; we will be primarily interested in the set $\psi(\SP(t))$ of words over $\Sigma_t$ that correspond to sorting plans of order $t$. The number of $t$-pop-stack-sortable permutations in $S_n$ is equal to $|\psi(\SP_n(t))|$, so the generating function whose rationality was demonstrated by Claesson and Gu{\dh}mundsson in \cite{ClaessonPop} is $\sum_{n\geq 1}|\psi(\SP_n(t))|z^n$. 

We say a segment $\sigma$ \dfn{contains} a segment $\sigma'$ if the word $\psi(\sigma)$ contains the word $\psi(\sigma')$ as a factor (i.e., a contiguous subword). For example, if $\sigma$ is the sorting plan on the right side of Figure~\ref{FigCox3} and $\sigma'$ is the segment on the left side of Figure~\ref{FigCox5}, then $\sigma$ contains $\sigma'$ because $\psi(\sigma)=0\,9\,10\,5\,5\,10\,4\,0$ and $\psi(\sigma')=9\,10\,5\,5\,10$.  

\begin{figure}[ht]
  \begin{center}{\includegraphics[height=1.811cm]{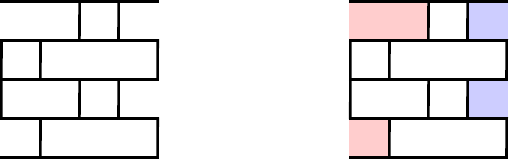}}
  \end{center}
  \caption{On the left is a segment of length $4$ and order $4$. On the right, the left-open (respectively, right-open) blocks of this segment have been shaded in pink (respectively, blue).}\label{FigCox5}
\end{figure}

Imagine enclosing a segment $\sigma$ of length $n$ and order $t$ inside a rectangle $R$ of width $n$ and height $t$. The connected components of the complement of $\sigma$ in the interior of $R$ are called the \dfn{blocks} of $\sigma$. A block is \dfn{left-open} if it touches $R$ on its left side but does not touch $\sigma$ on its left side. A block is \dfn{right-open} if it touches $R$, but not $\sigma$, on its right side. On the right side of Figure~\ref{FigCox5}, the left-open and right-open blocks of the segment have been shaded pink and blue, respectively. We say a segment is \dfn{bounded} if each of its blocks has length at most $3$. 

An \dfn{operation array} is a segment $\sigma$ such that $\psi(\sigma)$ starts and ends with $0$ (equivalently, $\sigma$ has no left-open or right-open blocks). Suppose we are given an operation array $\sigma$ of length $n$ and order $t$. Place the identity permutation $123\cdots n$ below $\sigma$. Now fill the rows of $\sigma$, one at a time, from bottom to top. At each step, copy the numbers in the $(k+1)^\text{th}$ row into the $k^\text{th}$ row and then reverse the numbers within each block in the $k^\text{th}$ row. Let $T$ be the resulting array of blocks and numbers, excluding the identity permutation at the bottom. The array $T$ is called the \dfn{semitrace} of $\sigma$ (see Figure~\ref{FigCox6}). Consider a pair of integers $(a,b)$ with $1\leq a<b\leq n$. Let $X_{a,b}(T)$ be the collection of blocks in $T$ that contain $a$ or $b$ and do not lie in the first row of $T$; let $X_{a,b}(\sigma)$ be the corresponding collection of blocks in $\sigma$. Let $\sigma_{a,b}$ be the smallest segment contained in $\sigma$ that includes all of the blocks in $X_{a,b}(\sigma)$. We say the pair $(a,b)$ is a \dfn{violating pair} of the semitrace $T$ if there is some row of $T$ that either contains $a$ and $b$ in the same block with $a$ immediately before $b$ or contains $a$ and $b$ in different blocks with $b$ immediately before $a$. We say the segment $\sigma_{a,b}$ is \dfn{forbidden} if $(a,b)$ is a violating pair in $T$. Claesson and Gu{\dh}mundsson showed that whether or not the segment $\sigma_{a,b}$ is forbidden only depends on the segment $\sigma_{a,b}$ itself; it does not depend on the pair $(a,b)$ or the operation array $\sigma$ in which $\sigma_{a,b}$ is embedded. 

\begin{figure}[ht]
  \begin{center}{\includegraphics[height=2.373cm]{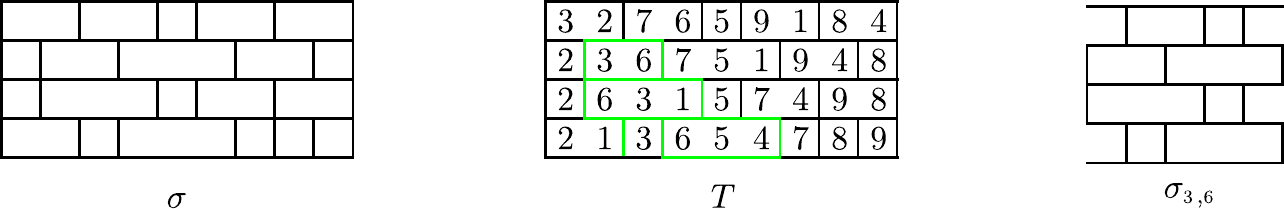}}
  \end{center}
  \caption{An operation array $\sigma$ (left) and its corresponding semitrace $T$ (middle). The blocks in $X_{3,6}(T)$ are shown in green; the corresponding blocks in $\sigma$ form the set $X_{3,6}(\sigma)$. The segment $\sigma_{3,6}$ (right) is the smallest segment contained in $\sigma$ that includes all of the blocks in $X_{3,6}(\sigma)$. The violating pairs of $T$ are $(3,6)$ (because of the second row) and $(5,6)$ (because of the first row). Since $(3,6)$ is a violating pair of $T$, the segment $\sigma_{3,6}$ is forbidden.}\label{FigCox6}
\end{figure}

\begin{lemma}[{\cite[Proposition~3.9]{ClaessonPop}}]\label{LemCharacterizeSPs}
A segment $\sigma$ of order $t$ is a sorting plan if and only if the following conditions hold: 
\begin{itemize}
\item $\sigma$ is an operation array; 
\item every block of $\sigma$ that is not in the first row of $\sigma$ has length at most $3$; 
\item $\sigma$ does not contain any bounded forbidden segments. 
\end{itemize}
\end{lemma}

\begin{lemma}[{\cite[Lemma~3.12]{ClaessonPop}}]\label{LemFiniteBFS}
There are finitely many bounded forbidden segments of order $t$. 
\end{lemma}

To prove their main result, Claesson and Gu{\dh}mundsson employed the theory of formal languages; we will do the same. We recall the basic notions from this theory, referring the reader to \cite{Linz} for more information. 

Let $\mathcal A$ be a nonempty finite alphabet. A \dfn{language} over $\mathcal A$ is a collection of finite (possibly empty) words over $\mathcal A$. Given a language $\mathcal L$, let $\mathcal L^*$ denote the set of all finite words, including the empty word, that can be obtained by concatenating words from $\mathcal L$. The \dfn{concatenation} of two languages $\mathcal L_1$ and $\mathcal L_2$ is the language $\mathcal L_1\mathcal L_2=\{xy:x\in\mathcal L_1,y\in\mathcal L_2\}$. The \dfn{reverse} of a language $\mathcal L$ is $\rev(\mathcal L)=\{\rev(x):x\in\mathcal L\}$, where, for $a_1,\ldots,a_k\in\mathcal A$, we write $\rev(a_1\cdots a_k)=a_k\cdots a_1$. 

A language is \dfn{regular} if it is the set of words accepted by a deterministic finite automaton. The following lemma lists several standard properties of the collection of regular languages; we refer to \cite[Chapter 4]{Linz} for its proof.

\begin{lemma}\label{LemRegular}
Let $\mathcal A$ be an alphabet. Every finite language over $\mathcal A$ is regular. If $\mathcal L,\mathcal L_1,\ldots,\mathcal L_k$ are regular languages over $\mathcal A$, then \[\mathcal L^*,\quad \bigcup_{i=1}^k\mathcal L_i,\quad\bigcap_{i=1}^k\mathcal L_i,\quad\mathcal L_1\mathcal L_2,\quad\mathcal A^*\setminus\mathcal L,\quad\rev(\mathcal L)\] are all regular.
\end{lemma}

Let $\mathcal A_n$ denote the set of words over $\mathcal A$ of length $n$. The crucial fact that we need states that if $\mathcal L$ is a regular language, then the generating function $\sum_{n\geq 1}|\mathcal L\cap\mathcal A_n|z^n$ is rational (see, e.g., \cite[Proposition I.2]{Flajolet}). Claesson and Gu{\dh}mundsson showed that $\psi(\SP(t))$ is a regular language over the alphabet $\Sigma_t$. We have seen that the words in $\psi(\SP(t))$ are in bijection with $t$-pop-stack-sortable permutations, so the regularity of $\psi(\SP(t))$ implies the rationality of the generating function that counts $t$-pop-stack-sortable permutations in symmetric groups.

We are almost ready to prove Theorem~\ref{ThmCox4}. We just need to discuss one additional concept. Say a segment $\sigma$ is \dfn{symmetric} if the word $\psi(\sigma)$ is a palindrome. In other words, a segment is symmetric if it is left unchanged when we reflect it through a central vertical axis. Observe that a segment is symmetric if and only if each of its rows is a symmetric segment of order $1$. Define a \dfn{type-$B$ sorting plan} to be a symmetric sorting plan of even length. 

\begin{lemma}\label{LemCox8}
Let $\sigma$ be the sorting plan of order $t$ of a $t$-pop-stack-sortable permutation $w\in S_{2n}$. Then $\sigma$ is a type-$B$ sorting plan if and only if $w\in B_n$. 
\end{lemma}
\begin{proof}
Suppose $w\in B_n$. We have seen that  $\Pop_{B_n}$ is the restriction of $\Pop_{S_{2n}}$ to $B_n$. Therefore, the permutations in the rows of the sorting trace of $w$ are $w,\Pop_{B_n}(w),\Pop_{B_n}^2(w),\ldots,\Pop_{B_n}^{t-1}(w)$. If $k\in[t]$ and $\delta_1,\ldots,\delta_r$ are the descending runs (from left to right) of $\Pop_{B_n}^{k-1}(w)$, then the length of $\delta_i$ is equal to the length of $\delta_{r+1-i}$ for each $i\in[r]$ (because $\Pop_{B_n}^{k-1}(w)\in B_n$). In other words, the $k^\text{th}$ row of $\sigma$ is a symmetric segment of length $2n$ and order $1$. As this is true for every $k\in [t]$, $\sigma$ must be a type-$B$ sorting plan. 

Conversely, suppose $\sigma$ is a type-$B$ sorting plan. We will prove by backward induction on $k$ that $\Pop^{k-1}(w)\in B_n$ for all $k\in[t+1]$; setting $k=1$ will then show that $w\in B_n$. We certainly have $\Pop^t(w)=e\in B_n$, so let us choose $k\in [t]$ and assume we have already proven that $\Pop^k(w)\in B_n$. Let $m_1,\ldots,m_r$ be the lengths of the descending runs (from left to right) of $\Pop^{k-1}(w)$. Let $\lambda$ be the unique layered permutation in $S_{2n}$ whose descending runs (from left to right) have lengths $m_1,\ldots,m_r$. By the discussion in Section~\ref{Subsec:PopMap}, we have $\Pop^k(w)=\Pop^{k-1}(w)\lambda$. Since $\sigma$ is symmetric and $m_1,\ldots,m_r$ are the lengths of the blocks in the $k^\text{th}$ row of $\sigma$, we have $m_i=m_{r+1-i}$ for all $i\in[r]$. It follows that $\lambda\in B_n$. Thus, $\Pop^{k-1}(w)=\Pop^k(w)\lambda^{-1}\in B_n$.  
\end{proof}

\begin{proof}[Proof of Theorem~\ref{ThmCox4}]
If $t=0$, then $\displaystyle \sum_{n\geq 1}\left|\Pop_{B_n}^{-t}(e)\right|z^n=\dfrac{z}{1-z}$ is rational, so we may assume $t\geq 1$. 
Given a palindromic word $x=x_1\cdots x_{2n}\in\Sigma_t^{2n}$, let $\half(x)=x_1\cdots x_n\in\Sigma_t^n$. Note that $x$ is uniquely determined by $\half(x)$ since $x=\half(x)\rev(\half(x))$. Let $\SP^B(t)$ be the set of type-$B$ sorting plans of order $t$. Each of the words in $\psi(\SP^B(t))$ is a palindrome of even length, so it makes sense to define the language $\half(\psi(\SP^B(t)))=\{\half(\psi(\sigma)):\sigma\in\SP^B(t)\}$. 

We have seen that the map sending each $t$-pop-stack-sortable permutation to its sorting plan is a bijection from the set of $t$-pop-stack-sortable permutations in $S_{2n}$ to $\SP_{2n}(t)$. It follows from Lemma~\ref{LemCox8} that the number of $t$-pop-stack-sortable elements of $B_n$ is equal to the number of type-$B$ sorting plans of length $2n$ and order $t$. This is also equal to the number of words of length $n$ in the language $\half(\psi(\SP^B(t)))$. Therefore, in order to prove that the generating function $\displaystyle \sum_{n\geq 1}\left|\Pop_{B_n}^{-t}(e)\right|z^n$ is rational, it suffices to show that $\half(\psi(\SP^B(t)))$ is a regular language. 

Let $U$ be the set of words $u\in\Sigma_t^*$ such that $\psi^{-1}(u)$ and $\psi^{-1}(\rev(u))$ do not contain any bounded forbidden segments while $\psi^{-1}(u\rev(u))$ does contain a bounded forbidden segment. Consider the following six properties that a word $x\in\Sigma_t^*$ may or may not have:
\begin{enumerate}[(i)]
\item\label{B1} $x$ begins with the letter $0$;
\item\label{B2} every block of $\psi^{-1}(x)$ that is not in the first row has length at most $3$;
\item\label{B3} every right-open block of $\psi^{-1}(x)$ that is not in the first row has length at most $1$;
\item\label{B4} $\psi^{-1}(x)$ does not contain any bounded forbidden segments;  
\item\label{B5} $\psi^{-1}(\rev(x))$ does not contain any bounded forbidden segments;  
\item\label{B6} no suffix of $x$ is in $U$. 
\end{enumerate}

It follows from Lemma~\ref{LemCharacterizeSPs} that a word $x\in\Sigma_t^*$ satisfies all six of these properties if and only if $\psi^{-1}(x\rev(x))$ is a sorting plan. Indeed, saying $x$ satisfies \eqref{B1} is equivalent to saying that $\psi^{-1}(x\rev(x))$ is an operation array. Saying $x$ satisfies \eqref{B2} and \eqref{B3} is equivalent to saying that every block of $\psi^{-1}(x\rev(x))$ that is not in the first row of $\psi^{-1}(x\rev(x))$ has length at most $3$. Finally, saying that $x$ satisfies \eqref{B4}, \eqref{B5}, and \eqref{B6} is equivalent to saying that $\psi^{-1}(x\rev(x))$ does not contain any bounded forbidden segments. Consequently, $\half(\psi(\SP^B(t)))$ is the set of words satisfying all six of the above properties. For each $\text{P}\in\{\text{i},\text{ii},\text{iii},\text{iv},\text{v},\text{vi}\}$, let $\mathcal L_\text{P}$ be the language of words over $\Sigma_t^*$ that satisfy the property ($\text{P}$). Since, by Lemma~\ref{LemRegular}, the intersection of a finite collection of regular languages is regular, our proof will be complete once we demonstrate that each of the languages $\mathcal L_{\text{P}}$ is regular. 

We make repeated tacit use of Lemma~\ref{LemRegular} in this paragraph. We have $\mathcal L_{\text{i}}=\{0\}\Sigma_t^*$, so $\mathcal L_{\text{i}}$ is regular. Let $Z_i$ be the set of letters $m\in\Sigma_t$ such that, when $m$ is written in binary as $x_t+2x_{t-1}+2^2x_{t-2}+\cdots+2^{t-1}x_1$, we have $x_i=1$. Equivalently, $Z_i$ is the set of elements $m$ of $\Sigma_t$ such that the  $i^\text{th}$ entry of the bar code $\psi^{-1}(m)$ is a blank space. We have \[\mathcal L_{\text{ii}}=\bigcap_{i=2}^t\left(\Sigma_t^*\setminus\left(\Sigma_t^*Z_iZ_iZ_i\Sigma_t^*\right)\right)\quad\text{and}\quad\mathcal L_{\text{iii}}=\bigcap_{i=2}^t\left(\Sigma_t^*\setminus\left(\Sigma_t^*Z_iZ_i\right)\right),\] so $\mathcal L_{\text{ii}}$ and $\mathcal L_{\text{iii}}$ are regular. Let $\mathcal F_t$ be the set of bounded forbidden segments of order $t$, which is finite by Lemma~\ref{LemFiniteBFS}. For each $f\in \mathcal F_t$, the set of words $x$ such that $\psi^{-1}(x)$ does not contain $f$ is $\Sigma^*\setminus\left(\Sigma_t^* \{\psi(f)\}\Sigma_t^*\right)$. Therefore, the language $\mathcal L_{\text{iv}}=\bigcap_{f\in\mathcal F_t}\left(\Sigma_t^*\setminus\left(\Sigma_t^* \{\psi(f)\}\Sigma_t^*\right)\right)$ is regular. Furthermore, $\mathcal L_{\text{v}}$ is regular because it is the reverse of $\mathcal L_{\text{iv}}$. 

We are left with the task of proving that $\mathcal L_{\text{vi}}$ is regular. Let $|x|$ denote the length of a word $x$. Since the set $\mathcal F_t$ of bounded forbidden segments of order $t$ is finite by Lemma~\ref{LemFiniteBFS}, there exists a positive integer $K$ such that $|\psi(f)|\leq K$ for all $f\in\mathcal F_t$. Let $U'$ be the set of words $u\in U$ such that no proper suffix of $u$ is in $U$. Consider $u\in U'$. Let $a$ be the first letter of $u$, and write $u=au'$. Note that $u'\not\in U$ because $u\in U'$. Since $u\in U$, we can write $u=vv'$ and $\rev(u)=y'y$ so that $v'$ and $y'$ are nonempty and $\psi^{-1}(v'y')\in\mathcal F_t$. Either $v$ or $y$ must be empty since, otherwise, $u'$ would be in $U$. Therefore, $|u|=\max\{|v'|,|y'|\}<|v'y'|\leq K$. This proves that every word in $U'$ has length at most $K$, so $U'$ is finite. In particular, $U'$ is a regular language by Lemma~\ref{LemRegular}. Note that a word $x$ is in $\mathcal L_{\text{vi}}$ if and only if no suffix of $x$ is in $U'$. In symbols, this says that $\mathcal L_{\text{vi}}=\Sigma_t^*\setminus(\Sigma_t^*U')$, so $\mathcal L_{\text{vi}}$ is regular by Lemma~\ref{LemRegular}.
\end{proof}

\section{$t$-Pop-Stack-Sortable Elements in Type $\widetilde A$}\label{SecAffine}

For $n\geq 1$, an \dfn{affine permutation} of size $n$ is a bijection $w:\mathbb Z\to\mathbb Z$ such that 
\begin{equation}\label{EqCox3}
w(i+n)=w(i)+n\quad\text{for all }i\in\mathbb Z
\end{equation} and 
\begin{equation}\label{EqCox2}
\sum_{i=1}^nw(i)=\binom{n+1}{2}.
\end{equation} The set $\widetilde S_n$ of affine permutations of size $n$ forms a group under composition; it is a Coxeter group of type $\widetilde A_{n-1}$. The simple generators are $\widetilde s_1,\ldots,\widetilde s_n$, where $\widetilde s_i$ is the affine permutation that swaps $i+mn$ and $i+mn+1$ for all $m\in\mathbb Z$ and fixes all other elements of $\mathbb Z$. The simple generator $\widetilde s_i$ is a right descent of an affine permutation $w$ if and only if $w(i)>w(i+1)$. Furthermore, $\ell(w)$ is equal to the number of pairs $(i,j)\in[n]\times\mathbb Z$ such that $i<j$ and $w(i)>w(j)$.  

It will be useful to consider the one-line notation of a bijection $w:\mathbb Z\to\mathbb Z$, which is simply the bi-infinite word $\cdots w(-2)w(-1)w(0).w(1)w(2)\cdots$. The decimal point between $w(0)$ and $w(1)$ is meant to indicate which letters are indexed by which integers. For example, $\cdots(-1)0.123\cdots$ represents the identity element $e$ of $\widetilde S_n$, while $\cdots(-1)01.23\cdots$ represents the bijection given by $i\mapsto i+1$, which is not an affine permutation because it fails to satisfy \eqref{EqCox2}. 

A \dfn{descending run} of an affine permutation $w$ is a maximal consecutive decreasing subsequence of $w$. We say an affine permutation $w\in \widetilde S_n$ is \dfn{layered} if there exists $k\in\{0,\ldots,n-1\}$ such that $(w(k+1)-k)(w(k+2)-k)\cdots(w(k+n)-k)$ is a layered permutation in $S_n$. In other words, we can think of a layered affine permutation as an infinite direct sum of decreasing permutations. Just as for symmetric groups, one can show that an affine permutation $w\in\widetilde S_n$ is layered if and only if $w=w_0(D_R(w))$. Hence, for an arbitrary $w\in\widetilde S_n$, we can compute $\Pop_{\widetilde S_n}(w)$ (using \eqref{EqCox1}) by multiplying $w$ on the left by the unique layered affine permutation that has the same right descent set as $w$. Using this description, it is straightforward to check that $\Pop_{\widetilde S_n}(w)$ is obtained by reversing all of the descending runs of $w$ while keeping entries in different descending runs in the same relative order.

\begin{lemma}\label{LemCox9}
Let $w\in\widetilde S_n$. Every descending run of $\Pop(w)$ has length at most $3$. 
\end{lemma}

\begin{proof}
Let $\Pop(w)=v$, and suppose, by way of contradiction, that there exists $i\in\mathbb Z$ such that $v(i)>v(i+1)>v(i+2)>v(i+3)$. For $j\in\{0,1,2,3\}$, let $\delta_j$ be the descending run of $w$ that contains the entry $v(i+j)$. Since $v$ is obtained by reversing the descending runs of $w$, the descending runs $\delta_0,\delta_1,\delta_2,\delta_3$ are distinct and appear consecutively (in this order) in $w$. This implies that the only entry in $\delta_1$ is $v(i+1)$ and that the only entry in $\delta_2$ is $v(i+2)$. Since $v(i+1)>v(i+2)$, this contradicts the fact that $\delta_1$ and $\delta_2$ are distinct descending runs. 
\end{proof}

Our goal in this section is to prove Theorem~\ref{ThmCox5}, which states that the generating function $\displaystyle \sum_{n\geq 1}\left|\Pop_{\widetilde S_n}^{-t}(e)\right|z^n$ is rational. First, we should check that this generating function is even well-defined! In other words, we should verify that for each $t\geq 0$, there are only finitely many $t$-pop-stack-sortable affine permutations in $\widetilde S_n$. This follows from our proof of Theorem~\ref{ThmCox1}. In that proof, we showed that $\left|O_{\Pop_{\widetilde S_n}}(w)\right|\geq\ell(w)/K+1$ for all $w\in \widetilde S_n$, where $K=\max\limits_{J\in\mathcal N(\widetilde S_n,S)}\ell(w_0(J))$. Since there are only finitely many elements of $\widetilde S_n$ of each fixed length, there are only finitely many $t$-pop-stack-sortable affine permutations in $\widetilde S_n$. 

We are going to make use of the ideas from the previous section concerning sorting traces, sorting plans, segments, and semitraces, but we need to modify them for the affine setting. Fix $t\geq 1$, and consider a $t$-pop-stack-sortable affine permutation $w\in\widetilde S_n$. Write out the one-line notations of $w,\Pop(w),\Pop^2(w),\ldots,\Pop^{t-1}(w)$, with $\Pop^k(w)$ directly below $\Pop^{k-1}(w)$ for each $k\in[t-1]$. Draw boxes around the descending runs of each of these affine permutations. The resulting bi-infinite array of numbers and boxes is the \dfn{affine sorting trace} of $w$. Deleting the numbers in the sorting trace produces a bi-infinte array of boxes called the \dfn{affine sorting plan} of $w$. We call $n$ and $t$ the \dfn{period} and \dfn{order}, respectively, of the affine sorting trace and the affine sorting plan. Note that an affine sorting plan of period $n$ also has period $dn$ for each positive integer $d$. The columns of an affine sorting plan are indexed by $\mathbb Z$ (so an affine sorting plan can change when it is shifted). Figure~\ref{FigCox7} shows the affine sorting trace and affine sorting plan of order $3$ of the affine permutation $w\in\widetilde S_5$ with $w(1)=0$, $w(2)=3$, $w(3)=2$, $w(4)=6$, $w(5)=4$. Denote by $\widetilde{\SP}_n(t)$ the set of all affine sorting plans of period $n$ and order $t$ (i.e., the set of all affine sorting plans of $t$-pop-stack-sortable permutations in $\widetilde S_n$). Let $\widetilde{\SP}(t)=\bigcup_{n\geq 1}\widetilde{\SP}_n(t)$. As in the non-affine case, boxes in affine sorting traces and affine sorting plans are called \dfn{blocks}.

\begin{figure}[ht]
  \begin{center}{\includegraphics[height=3.341cm]{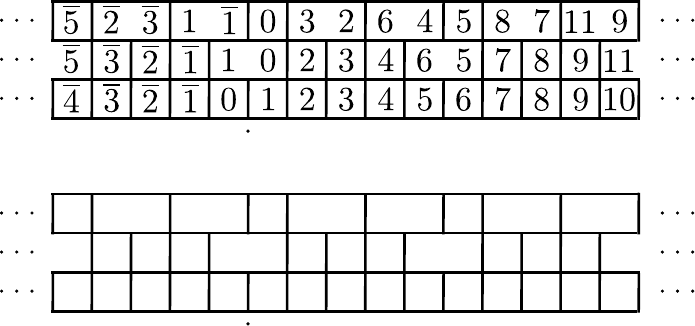}}
  \end{center}
  \caption{The affine sorting trace (top) and affine sorting plan (bottom) of order $3$ of the affine permutation $w\in\widetilde S_5$ satisfying $w(1)=0$, $w(2)=3$, $w(3)=2$, $w(4)=6$, $w(5)=4$. In each diagram, notice the decimal point between the column indexed by $0$ and the column indexed by $1$.}\label{FigCox7}
\end{figure}

Bar codes have the same meaning that they did in the previous section, and we still have a natural bijection $\psi$ between bar codes of order $t$ and elements of the alphabet $\Sigma_t=\{0,\ldots,2^t-1\}$. Segments also have the same meaning as in the previous section, except we now allow them to be bi-infinite (in which case, the columns are indexed by $\mathbb Z$). Blocks of segments are defined in the obvious way (blocks are now permitted to be infinitely long). In this setting, $\psi$ extends to an injection from the set of segments of order $t$ to the set of finite or bi-infinite words over $\Sigma_t$. We say a bi-infinite segment $\sigma$ is \dfn{$n$-periodic} if $\psi(\sigma)$ is an $n$-periodic word. As before, we say a segment $\sigma$ \dfn{contains} a segment $\sigma'$ if $\psi(\sigma)$ contains $\psi(\sigma')$ as a factor. Finally, we say a segment $\sigma$ is \dfn{non-Escher}\footnote{This clever terminology is stolen from \cite{Adin}, where it was used in a similar but different manner. The fact that an affine sorting plan is non-Escher corresponds to the fact that an affine permutation in $\widetilde S_n$ cannot have the entire set $S=\{\widetilde s_1,\ldots,\widetilde s_n\}$ as its right descent set, contrary to M. C. Escher's classical pieces of art that portray staircases perpetually descending and somehow looping back on themselves.} if it does not contain any infinitely long blocks. 

Now suppose we are given a non-Escher bi-infinite segment $\sigma$ of order $t$, which we assume has its columns indexed by the integers. Place the identity affine permutation below $\sigma$ so that for each $i\in\mathbb Z$, the number $i$ is below column $i$. Fill the rows of $\sigma$ from bottom to top. At each step, copy the numbers in the $(k+1)^\text{th}$ into the $k^\text{th}$ row and then reverse the numbers within each block in the $k^\text{th}$ row. The resulting array $T$, excluding the identity affine permutation at the very bottom, is called the \dfn{affine semitrace} of $\sigma$. Given a pair of integers $(a,b)$ with $a<b$, we define $\sigma_{a,b}$ just as in the previous section. As before, we say the pair $(a,b)$ is a \dfn{violating pair} of the affine semitrace $T$ if there is some row of $T$ that either contains $a$ and $b$ in the same block with $a$ immediately before $b$ or contains $a$ and $b$ in different blocks with $b$ immediately before $a$. We say the segment $\sigma_{a,b}$ is \dfn{forbidden} if $(a,b)$ is a violating pair in $T$. 

Note that whether or not a finite-length segment is forbidden is a local property; this means that a segment of finite length is forbidden in the affine setting if and only if it is forbidden in the sense of the preceding section. Thus, the set $\mathcal F_t$ of bounded forbidden segments of order $t$ is the same as it was in the previous section. In particular, bounded forbidden segments are of finite length, and the set $\mathcal F_t$ is finite by Lemma~\ref{LemFiniteBFS}. 

The map sending each affine permutation to its affine sorting plan is a bijection from $\Pop_{\widetilde S_n}^{-t}(e)$ to $\widetilde{\SP}_n(t)$. Indeed, each $w\in\widetilde S_n$ is determined by its affine sorting plan $\sigma$ because $w$ is the sequence of numbers in the first row of the affine semitrace of $\sigma$.

The following lemma serves as an affine version of Lemma~\ref{LemCharacterizeSPs}. 

\begin{lemma}\label{LemCox10}
A segment $\sigma$ of order $t$ is in $\widetilde{\SP}_n(t)$ if and only if the following conditions hold: 
\begin{itemize}
\item $\sigma$ is $n$-periodic and non-Escher; 
\item every block of $\sigma$ that is not in the first row of $\sigma$ has length at most $3$; 
\item $\sigma$ does not contain any bounded forbidden segments. 
\end{itemize}
\end{lemma}

\begin{proof}
Suppose first that $\sigma$ is in $\widetilde{\SP}_n(t)$. This means that $\sigma$ is the sorting plan of a $t$-pop-stack-sortable affine permutation $w\in\widetilde S_n$. The lengths of the blocks in the $k^\text{th}$ row of $\sigma$ are the lengths of the descending runs of $\Pop_{\widetilde S_n}^{k-1}(w)$. The condition \eqref{EqCox3} ensures that, for each $k\in[t]$, the right descent set of $\Pop_{\widetilde S_n}^{k-1}(w)$ is not the entire set $S=\{\widetilde s_1,\ldots,\widetilde s_n\}$, so $\sigma$ is non-Escher. The condition \eqref{EqCox3} also guarantees that $\sigma$ is $n$-periodic. Furthermore, Lemma~\ref{LemCox9} implies that every block of $\sigma$ that is not in the first row has length at most $3$. Because $w$ is $t$-pop-stack-sortable, the affine trace of $w$ is the same as the affine semitrace of $\sigma$. The affine trace of $w$ cannot contain a violating pair since its blocks are constructed by putting boxes around the descending runs of the affine permutations in each row. It follows that $\sigma$ cannot contain a forbidden segment. In particular, $\sigma$ does not contain any bounded forbidden segments.

To prove the converse, assume $\sigma$ satisfies the three bulleted conditions in the statement of the lemma. Construct the affine semitrace $T$ of $\sigma$. Consider the step when we fill the $k^\text{th}$ row of the semitrace by copying the numbers in the $(k+1)^\text{th}$ row into the $k^\text{th}$ row and then reversing the numbers within each block in the $k^\text{th}$ row. If we already know that the entries in the $(k+1)^\text{th}$ row form an affine permutation in $\widetilde S_n$, then the entries in the $k^\text{th}$ row must also form an affine permutation in $\widetilde S_n$. Indeed, it follows from the first two bulleted conditions that the entries in the $k^\text{th}$ row are obtained by multiplying the affine permutation in the $(k+1)^\text{th}$ row by a layered affine permutation. Let $w$ be the bounded affine permutation in the first row of $T$. It follows from the second bulleted condition that every segment contained in $\sigma$ is bounded, so it follows from the third bulleted condition that $\sigma$ does not contain any forbidden segments. Consequently, $T$ does not contain any violating pairs. This implies that $w$ is $t$-pop-stack-sortable and that $T$ is the affine trace of $w$. Hence, $\sigma\in\widetilde{\SP}_n(t)$. 
\end{proof}

Before proving our final theorem, we need one additional lemma about regular languages. Given a word $x=x_1\cdots x_n$ over a finite alphabet $\mathcal A$, define $\cyc(x)$ to be the cyclic shift $x_2\cdots x_nx_1$. For a language $\mathcal L\subseteq\mathcal A^*$, define $\cyc(\mathcal L)=\{\cyc(x):x\in \mathcal L\}$. 

\begin{lemma}[{\cite[Chapter 4, Exercise 20]{Linz}}]\label{LemCox11}
If $\mathcal L$ is a regular language over a finite alphabet $\mathcal A$, then so is $\cyc(\mathcal L)$. 
\end{lemma}

\begin{proof}[Proof of Theorem~\ref{ThmCox5}]
If $t=0$, then $\displaystyle \sum_{n\geq 1}\left|\Pop_{\widetilde S_n}^{-t}(e)\right|z^n=\dfrac{z}{1-z}$ is rational, so we may assume $t\geq 1$. Let $K=\max\{K'+1,5\}$, where $K'$ is the maximum length of a bounded forbidden segment of order $t$ (this is finite by Lemma~\ref{LemFiniteBFS}). To prove the theorem, it suffices to show that $\displaystyle \sum_{n\geq K}\left|\Pop_{\widetilde S_n}^{-t}(e)\right|z^n$ is rational.

Assume $n\geq K$. Suppose $w\in\widetilde S_n$, and let $T$ and $\sigma$ be the affine sorting trace and affine sorting plan, respectively, of $w$. The $i^\text{th}$ column of $T$ is the column of numbers whose top entry is $w(i)$. Recall the injection $\psi$ that sends segments of order $t$ to words over $\Sigma_t$. Also, recall that $\sigma$ can be seen as a bi-infinite sequence of bar codes; let $b_i$ be the bar code in $T$ that is immediately to the left of the $i^\text{th}$ column of $T$. Let $a_i\in\Sigma_t$ be the letter corresponding to $b_i$ (via its binary expansion). Define $\alpha_n(\sigma)$ to be the segment of length $n-1$ such that $\psi(\alpha_n(\sigma))=a_1a_2\cdots a_n$. For example, if $\sigma$ is the affine sorting plan in Figure~\ref{FigCox7}, then \[\alpha_5(\sigma)=\psi^{-1}(20404)=\begin{array}{l}\includegraphics[height=1.366cm]{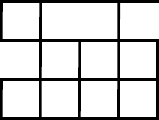}\end{array}.\] Notice that $\sigma$ can be reconstructed from $\alpha_n(\sigma)$ because it is $n$-periodic by Lemma~\ref{LemCox10}. Thus, $\alpha_n$ is an injection from $\widetilde{\SP}_n(t)$ into the set of segments of length $n-1$ and order $t$. The number of $t$-pop-stack-sortable affine permutations in $\widetilde S_n$ is equal to $|\widetilde{\SP}_n(t)|$, which is equal to $|\alpha_n(\widetilde{\SP}_n(t))|$, which is equal to $|\psi(\alpha_n(\widetilde{\SP}_n(t)))|$. 

Consider the language $\mathcal L=\bigcup_{n\geq K}\psi(\alpha_n(\widetilde{\SP}_n(t)))$ over $\Sigma_t$. In order to complete the proof of the theorem, it suffices to show that $\mathcal L$ is regular. To do this, we define $Y$ to be the set of finite words $y$ over $\Sigma_t$ such that:
\begin{enumerate}[(I)]
\item\label{Aff1} $y$ has length at least $K$;
\item\label{Aff2} the first row of $\psi^{-1}(y)$ has at least one vertical bar;
\item\label{Aff3} every block of $\psi^{-1}(y)$ that is not in the first row of $\psi^{-1}(y)$ has length at most $3$;
\item\label{Aff4} $\psi^{-1}(y)$ does not contain any bounded forbidden segments. 
\end{enumerate}
Let $\mathcal L'=\bigcap_{k=0}^{K-1}\cyc^k(Y)$, where $\cyc^k(Y)=\{\cyc^k(y):y\in Y\}$. 

We claim that $\mathcal L'=\mathcal L$. To see this, first suppose $x=x_1\cdots x_n\in\mathcal L'$. Since $x\in Y$, we have $n\geq K$ by \eqref{Aff1}. Let $\sigma=\psi^{-1}(\cdots xx.xxx\cdots)$ be the bi-infinite segment obtained by concatenating $\psi^{-1}(x)$ with itself infinitely many times. We want to show that $\sigma$ satisfies the three bulleted conditions in Lemma~\ref{LemCox10}; it will then follow that $\sigma\in\widetilde{\SP}_n(t)$ so that $x=\psi(\alpha_n(\sigma))$. The $n$-periodicity of $\sigma$ is clear from its definition, and the non-Escher property of $\sigma$ follows from \eqref{Aff2}, \eqref{Aff3}, and the fact that $n\geq K\geq 5$. If $\sigma$ has a block of length at least $4$ that is not in its first row, then there is some $k\in\{0,1,2,3,4\}$ such that $\psi^{-1}(\cyc^{-k}(x))$ has a block of length at least $4$ not in its first row. However, this contradicts property \eqref{Aff3} of $\cyc^{-k}(x)$, which is in $Y$ because $K\geq 5$. Finally, if $\sigma$ contains a bounded forbidden segment, then there is some $k\in\{0,1,\ldots,K'\}$ such that $\psi^{-1}(\cyc^{-k}(x))$ contains a bounded forbidden segment. This contradicts property \eqref{Aff4} of $\cyc^{-k}(x)$, which is in $Y$ because $K\geq K'+1$. Hence, $\mathcal L'\subseteq \mathcal L$. 

Now suppose $x\in\mathcal L$. Then $x=\psi(\alpha_n(\sigma))$, where $\sigma$ is the affine sorting plan of order $t$ of some $w\in\widetilde S_n$ with $n\geq K$. Choose $k\geq 0$, and let $y=\cyc^{-k}(x)$. The word $y$ has length $n$, so it satisfies \eqref{Aff1}. Notice that $\psi(\sigma)$ is the bi-infinite word $\cdots xx.xxx\cdots$, which is obtained by shifting $\cdots yy.yyy\cdots$ by $k$. Lemma~\ref{LemCox10} tells us that $\sigma$ is non-Escher, so $y$ must satisfy \eqref{Aff2}. The second bulleted item in Lemma~\ref{LemCox10} implies that $y$ satisfies \eqref{Aff3}. Similarly, the third bulleted item in Lemma~\ref{LemCox10} implies that $y$ satisfies \eqref{Aff4}. This proves that $y\in Y$. As $k$ was arbitrary, $x\in\bigcap_{k\geq 0}\cyc^k(Y)\subseteq\mathcal L'$. 

We have now established that $\mathcal L=\mathcal L'$. We are left with the task of proving that $\mathcal L'$ is a regular language. We make tacit use of Lemma~\ref{LemRegular}. For $\text{P}\in\{\text{I},\text{II},\text{III},\text{IV}\}$, let $\mathcal L_{\text{P}}$ be the set of words $y$ in $\Sigma_t^*$ satisfying property (P). Note that $\mathcal L_{\text{I}}$ is regular because it is equal to $(\Sigma_t)_K\Sigma_t^*$, where $(\Sigma_t)_K$ is the regular language consisting of all words in $\Sigma_t^*$ of length $K$. Let $Z_1$ be the set of letters $m\in\Sigma_t$ such that, when $m$ is written in binary as $x_t+2x_{t-1}+2^2x_{t-2}+\cdots+2^{t-1}x_1$, we have $x_1=1$. Equivalently, $m\in Z_1$ if and only if the bar code corresponding to $m$ starts with a blank space. Then $\mathcal L_{\text{II}}=\Sigma_t^*\setminus(Z_1^*)$ is regular. The languages $\mathcal L_{\text{III}}$ and $\mathcal L_{\text{IV}}$ are equal to the languages $\mathcal L_{\text{ii}}$ and $\mathcal L_{\text{iv}}$, respectively, from the proof of Theorem~\ref{ThmCox4} in Section~\ref{Sec:TypeB}; we saw in that proof that these languages are regular. We conclude that $\mathcal L_{\text{I}},\mathcal L_{\text{II}},\mathcal L_{\text{III}},\mathcal L_{\text{IV}}$ are regular languages, so their intersection $Y$ is regular as well. It follows from Lemma~\ref{LemCox11} that $\cyc^k(Y)$ is regular for each $k\in\{0,\ldots,K-1\}$. Hence, $\mathcal L'$ is regular. 
\end{proof}

\section{Further Directions}\label{SecConclusion}

In Section~\ref{Subsec:Further}, we mentioned two extensions of Coxeter pop-stack-sorting operators: one to other complete meet-semilattices and one to other semilattice congruences on weak orders of Coxeter groups. We explore the first of these extensions in \cite{DefantMeeting} and explore the second in \cite{DefantCoxeterStack}. Here, we mention some other potential avenues for future work. 

The authors of \cite{Asinowski} suggested considering the average size of the forward orbit of a permutation in $S_n$ under the pop-stack-sorting map. In \cite{DefantMonotonicity}, the current author conjectured that this average number of iterations is asymptotically equal to $n$, which is the maximum possible size of a forward orbit by Ungar's theorem. We believe that the same statement should hold for Coxeter groups of other classical types as well. In these cases, the maximum possible size of the forward orbit of an element is the Coxeter number of the group by Theorem~\ref{ThmCox1}. The Coxeter number of $B_n$ is $2n$, and the Coxeter number of $D_n$ is $2n-2$. 

\begin{conjecture}
As $n\to\infty$, we have \[\frac{1}{|B_n|}\sum_{w\in B_n}\left|O_{\Pop_{B_n}}(w)\right|\sim 2n\quad\text{and}\quad\frac{1}{|D_n|}\sum_{w\in D_n}\left|O_{\Pop_{D_n}}(w)\right|\sim 2n-2.\]
\end{conjecture} 

In Section~\ref{Subsec:Further}, we defined a map $\Pop_M:M\to M$, where $M$ is an arbitrary complete meet-semilattice. In Remark~\ref{RemCox1}, we defined the notion of a \emph{compulsive} map $f:M\to M$. We also exhibited a $6$-element lattice $M$ and a compulsive map $f:M\to M$ such that $\sup\limits_{x\in M}\left|O_f(x)\right|>\sup\limits_{x\in M}\left|O_{\Pop_M}(x)\right|$. It could be interesting to investigate which complete meet semilattices $M$ have the property that $\sup\limits_{x\in M}\left|O_f(x)\right|\leq\sup\limits_{x\in M}\left|O_{\Pop_M}(x)\right|$ for every compulsive map $f:M\to M$; Theorems~\ref{ThmCox1} and~\ref{ThmCox2} tell us that weak orders of Coxeter groups have this property. It is also natural to consider this question only for finite meet-semilattices, or even just for finite lattices.

\section{Acknowledgments}
The author thanks Henri M\"uhle for pointing out the connection between this work and the articles \cite{Muhle1, Muhle2}. He also thanks Nathan Williams for pointing out the connection between $\Pop$ and Brieskorn normal form. The author was supported by a Fannie and John Hertz Foundation Fellowship and an NSF Graduate Research Fellowship.

\end{document}